%% file: Modal_logic_of_set-theoretic_potentialism.tex
\documentclass{amsart}

\title[Set-theoretic potentialism]{The modal logic of set-theoretic potentialism and the potentialist maximality principles}

\author{Joel David Hamkins}
 \address[Joel David Hamkins]
          {Professor of Logic, University of Oxford, and
           Sir Peter Strawson Fellow in Philosophy, University College, Oxford}
\email{jhamkins@gc.cuny.edu}
\urladdr{http://jdh.hamkins.org}

\author{\Oystein\ Linnebo}
 \address[\Oystein\ Linnebo]
         {Department of Philosophy, IFIKK, University of Oslo, Postboks 1020 Blindern, 0315 Oslo, Norway}
\email{oystein.linnebo@ifikk.uio.no}
\urladdr{http://oysteinlinnebo.org}

\thanks{The authors are grateful to the Set-theoretic Pluralism network for support provided for their participation at the STP Symposium 2016 in Aberdeen, where this work was initiated. The research of the first author has been supported by grant \#69573-00 47 from the CUNY Research Foundation. Commentary can be made about this article on his blog at %http://jdh.hamkins.org/set-theoretic-potentialism.
\href{http://jdh.hamkins.org/set-theoretic-potentialism}{http://jdh.hamkins.org/set-theoretic-potentialism}.
}
\usepackage{latexsym,amsfonts,amsmath,amssymb,mathrsfs}
\usepackage[hidelinks]{hyperref}
\usepackage{enumitem}
\usepackage{tikz}
\include{MathMacrosJDH}
\newcommand\Val{\mathord{\rm Val}}
\usepackage[backend=bibtex,style=alphabetic,maxbibnames=15,maxcitenames=6,dateabbrev=false]{biblatex}
\renewcommand{\UrlFont}{\sffamily\small} % makes url text smaller (used only in bibliography?)
\renewbibmacro{in:}{\ifentrytype{article}{}{\printtext{\bibstring{in}\intitlepunct}}} % removes "In:" in article format.
\DeclareFieldFormat{url}{{\tiny{\UrlFont\url{#1}}}} % removes redundant "URL:" text, and just gives the url.
\DeclareFieldFormat{urldate}{% improves setting of urldate
  (version \thefield{urlday}\addspace%
  \mkbibmonth{\thefield{urlmonth}}\addspace%
  \thefield{urlyear}\isdot)}
\DeclareFieldFormat{eprint:arxiv}{%   JDH format for arxiv info in bibliography
  \ifhyperref
    {\href{http://arxiv.org/abs/#1}{%
        arXiv\addcolon\nolinkurl{#1}}\iffieldundef{eprintclass}{}{\UrlFont{\mkbibbrackets{\thefield{eprintclass}}}}}
    {arXiv\addcolon\nolinkurl{#1}\iffieldundef{eprintclass}{}{\UrlFont{\mkbibbrackets{\thefield{eprintclass}}}}}}
\bibliography{HamkinsBiblio,MathBiblio,WebPosts}

\begin{document}

\begin{abstract}
We analyze the precise modal commitments of several natural varieties of set-theoretic potentialism, using tools we develop for a general model-theoretic account of potentialism, building on those of Hamkins, Leibman and \Lowe~\cite{HamkinsLeibmanLoewe2015:StructuralConnectionsForcingClassAndItsModalLogic}, including the use of buttons, switches, dials and ratchets. Among the potentialist conceptions we consider are: rank potentialism (true in all larger $V_\beta$); Grothendieck-Zermelo potentialism (true in all larger $V_\kappa$ for inaccessible cardinals $\kappa$); transitive-set potentialism (true in all larger transitive sets); forcing potentialism (true in all forcing extensions); countable-transitive-model potentialism (true in all larger countable transitive models of \ZFC); countable-model potentialism (true in all larger countable models of \ZFC); and others. In each case, we identify lower bounds for the modal validities, which are generally either S4.2 or S4.3, and an upper bound of S5, proving in each case that these bounds are optimal. The validity of S5 in a world is a potentialist maximality principle, an interesting set-theoretic principle of its own. The results can be viewed as providing an analysis of the modal commitments of the various set-theoretic multiverse conceptions corresponding to each potentialist account.
\end{abstract}

\maketitle

\noindent
Set-theoretic potentialism is the view in the philosophy of mathematics that the universe of set theory is never fully completed, but rather unfolds gradually as parts of it increasingly come into existence or become accessible to us. On this view, the outer or upper reaches of the set-theoretic universe have merely potential rather than actual existence, in the sense that one can always form additional sets from that realm, as many as desired, but the task is never completed. For example, height potentialism is the view that the universe is never fully completed with respect to height: new ordinals come into existence as the known part of the universe grows ever taller. Width potentialism holds that the universe may grow outwards, as with forcing, so that already-existing sets can potentially gain new subsets in a larger universe. One commonly held view amongst set theorists is height potentialism combined with width actualism, whereby the universe grows only upward rather than outward, and so at any moment the part of the universe currently known to us is a rank initial segment $V_\alpha$ of the potential yet-to-be-revealed higher parts of the universe.\footnote{See e.g.~\cite{Zermelo:1930} (also discussed below),~\cite{Putnam:1967b} and~\cite{Hellman:1989}, and~\cite{Studd_ItConceptBiModal}. } Such a perspective might even be attractive to a Platonistically inclined large-cardinal set theorist, who wants to hold that there are many large cardinals, but who also is willing at any moment to upgrade to a taller universe with even larger large cardinals than had previously been mentioned. Meanwhile, the width-potentialist height-actualist view may be attractive for those who wish to hold a potentialist account of forcing over the set-theoretic universe $V$. On the height-and-width-potentialist view, one views the universe as growing with respect to both height and width.\footnote{See e.g.~\cite{Parsons:1983b} and~\cite{Linnebo:2013-PHS}.} A set-theoretic monist, in contrast, with an ontology having only a single fully existing universe, will be an actualist with respect to both width and height. The second author has described various potentialist views in~\cite{Linnebo:2013-PHS} and~\cite{LinneboShapiro2017:Actual-and-potential-infinity}.

Our project here is to analyze and understand more precisely the modal commitments of various set-theoretic potentialist views. We shall restrict ourselves in this analysis, however, to forms of potentialism based on \emph{classical} modal logic.\footnote{A stricter interpretation of potentialism can be shown to call for intuitionistic logic (without any commitment to the anti-realist assumptions typically used to support this logic)~\cite{LinneboShapiro2017:Actual-and-potential-infinity}. Analogous ideas come up in~\cite{Lear1977-setssem} and~\cite{Tait1998-ZermeloConceptionSets}.} After developing a general model-theoretic account of the semantics of potentialism in section~\ref{Section.semantics-of-potentialism} and providing tools in section~\ref{Section.tools} for establishing both lower and upper bounds on the modal validities for various kinds of potentialist contexts, we shall use those tools in section~\ref{Section.set-theoretic-potentialism} to settle exactly the propositional modal validities for several natural kinds of set-theoretic height and width potentialism.

The various kinds of set-theoretic potentialism validate in each case a specific corresponding modal theory, which expresses the fundamental principles or potentialist philosophy of that conception. Set-theoretic rank-potentialism and Grothendieck-Zermelo potentialism, for example, validate precisely S4.3 and exhibit what could be described as a character of linear inevitability: the central axiom (.3) can be interpreted as a linearity assertion; the finite pre-linear Kripke frames are complete for S4.3; and the worlds in these potentialist systems are linearly ordered. Forcing potentialism and countable-transitive model potentialism, in contrast, exhibit a character of directed convergence, with the validities being exactly S4.2, whose central axiom expresses the idea that any two possibilities can be brought together again in a common further possibility, but without any presumption of linearity. Some systems validate S5, whose central axiom $\possible\necessary\varphi\to\varphi$ expresses maximality or completeness, where any statement that could become permanent has already been made permanent. Meanwhile, the potentialist systems validating only S4, as in \cite{Hamkins:The-modal-logic-of-arithmetic-potentialism, HamkinsWoodin:The-universal-finite-set}, exhibit a far more radical branching character, one in which fundamental bifurcations in possibility are revealed as the universe unfolds in one manner as opposed to another incompatible manner.

Thus, there is a watershed between S4 on the one hand and S4.2 and above on the other. Only in the latter group of theories does the potentialist translation of theorem \ref{Theorem.Potentialist-translation} (and the mirroring theorem of \cite{Linnebo:2013-PHS}) work as it should; this breaks down when the logic is below S4.2. The first author argues in \cite[section~7]{Hamkins:The-modal-logic-of-arithmetic-potentialism} that the convergent forms of potentialism can be seen as implicitly actualist, although the second author counters this with concerns of higher-order logic \cite[section 7]{LinneboShapiro2017:Actual-and-potential-infinity}.

Let us also briefly mention the strong affinities between set-theoretic potentialism and set-theoretic pluralism, particularly with the various set-theoretic multiverse conceptions currently in the literature, such as those of the first author~\cite{GitmanHamkins2010:NaturalModelOfMultiverseAxioms, Hamkins2011:TheMultiverse:ANaturalContext, Hamkins2012:TheSet-TheoreticalMultiverse, Hamkins2014:MultiverseOnVeqL}. Potentialists may regard themselves mainly as providing an account of truth ultimately for a single universe, gradually revealed, the limit of their potentialist system. Nevertheless, the universe fragments of their potentialist account can often naturally be taken as universes in their own right, connected by the potentialist modalities, and in this way, every potentialist system can be viewed as a  multiverse. Indeed, the potentialist systems we analyze in this article---including rank potentialism, forcing potentialism, generic-multiverse potentialism, countable-transitive-model potentialism, countable-model potentialism---each align with corresponding natural multiverse conceptions. Because of this, we take the results of this article as providing not only an analysis of the modal commitments of set-theoretic potentialism, but also an analysis of the modal commitments  of various particular set-theoretic multiverse conceptions. Indeed, one might say that it is possible ({\it ahem}), in another world, for this article to have been entitled, ``The modal logic of various set-theoretic multiverse conceptions.''

\section{The semantics of potentialism}\label{Section.semantics-of-potentialism}

Although we are motivated by the case of set-theoretic potentialism, the potentialist idea itself is far more general, and can be carried out in a general model-theoretic context. For example, the potentialist account of arithmetic is deeply connected with the classical debates surrounding potential as opposed to actual infinity, and indeed, perhaps it is in those classical debates where one finds the origin of potentialism. More generally, one can provide a potentialist account of truth in the context of essentially any kind of structure in any language or theory.

Allow us therefore to lay out a general model-theoretic account of what we call the \emph{semantics of potentialism}.
Suppose that $\mathcal{W}$ is a collection of structures in a common signature $\mathcal{L}$---we shall call them \emph{worlds}---and that $\mathcal{W}$ is equipped with an accessibility relation, making it a Kripke model of $\mathcal{L}$-structures. We say that $\mathcal{W}$ is a {\df potentialist system}, if the accessibility relation is reflexive and transitive and if, furthermore, whenever one world accesses another in $\mathcal{W}$, then the domain of the first world is contained within that of the second. Thus, a potentialist system is a reflexive transitive Kripke model of $\mathcal{L}$-structures, in which the accessibility relation is inflationary in the domains of the structures, so that the individuals of a world continue to exist in any accessed new world. The main idea of potentialism is that any particular world $W$ regards the individuals of $W$ as the actually existing individuals or objects, and those in the larger yet-to-be-accessed worlds have merely a potential existence; but we may regard those potentially existing objects as actually existing, if we should simply move to the context and perspective of those larger worlds.

We shall be principally interested in the case where $\mathcal{L}$ is a first-order language and these are first-order structures, such as in the case of set-theoretic potentialism. Meanwhile, the ideas do readily generalize to nearly any kind of structure, using infinitary or higher-order logics, among others.

An important case of potentialism occurs when the accessibility relation of $\mathcal{W}$ is precisely the substructure relation on the family of structures. In particular, for any class $\mathcal{W}$ of $\mathcal{L}$-structures, we may place the {\df potentialist} accessibility relation on $\mathcal{W}$, by which world $W$ accesses $U$ if and only if $W$ is a substructure of $U$, so the domain of $W$ is contained within that of $U$ and they agree on the atomic facts about individuals in $W$. In this way, any class of $\mathcal{L}$-structures can be viewed as a potentialist system.

The semantics of potentialism for any potentialist system is defined, of course, in the usual Kripkean manner for assertions in the potentialist language $\mathcal{L}^\Diamond$, which augments $\mathcal{L}$ with the modal operators $\possible$ and $\necessary$. Namely, we define recursively what it means to say that an $\mathcal{L}^\Diamond$-formula $\varphi$ is true at a world $W\in\mathcal{W}$ with parameters $a_0,\dots,a_n\in W$, written:
 $$W\satisfies_{\mathcal{W}}\varphi(a_0,\dots,a_n).$$
Specifically, an atomic assertion $\varphi(a_0,\dots,a_n)$ is true at world $W$ in $\mathcal{W}$ for parameters $a_0,\dots,a_n\in W$ just in case the corresponding atomic fact is true in $W$ in the usual Tarskian sense of $\mathcal{L}$-satisfaction. The truth definition extends through Boolean combinations in the usual Tarskian manner, so that $W\satisfies_{\mathcal{W}}(\varphi\wedge\psi)(a_0,\dots,a_n)$ just in case both $\varphi(a_0,\dots,a_n)$ and $\psi(a_0,\dots,a_n)$ are true separately at world $W$ in $\mathcal{W}$; and $\neg\varphi(a_0,\dots,a_n)$ is true at world $W$ just in case $\varphi(a_0,\dots,a_n)$ is not true at $W$. Quantifiers are handled by quantifying only over the current actual individuals, that is, over the objects of the current possible world, so that $W\satisfies_{\mathcal{W}}\exists x\, \varphi(x,a_0,\dots,a_n)$ just in case there is some $a\in W$ for which $W\satisfies_{\mathcal{W}}\varphi(a,a_0,\dots,a_n)$. The modal semantics are determined by the potentialist accessibility relation, so that $W\satisfies_{\mathcal{W}}\possible\varphi(a_0,\dots,a_n)$ if and only if $W$ can access some world $U\in\mathcal{W}$ for which $U\satisfies_{\mathcal{W}}\varphi(a_0,\dots,a_n)$. Similarly, $\necessary\varphi(a_0,\dots,a_n)$ is true at world $W$ just in case $\varphi(a_0,\dots,a_n)$ is true at all such accessed worlds $U$ in $\mathcal{W}$. Notice that since the domains of the possible worlds are inflationary along the accessibility relation, there is no need to worry about parameters ceasing to exist in connection with the clauses just mentioned for the modal operators; thus, we avoid what can otherwise often be a troublesome matter for modal semantics. In this way, any collection $\mathcal{W}$ of $\mathcal{L}$-structures in the same language determines a corresponding potentialist semantics in the modal-equipped language $\mathcal{L}^\Diamond$ using the potentialist accessibility relation.

For sufficiently coherent collections of structures $\mathcal{W}$, the potentialist approach to truth can interact in a meaningful way with the usual Tarskian account of truth in a kind of limit structure. Namely, we define that a potentialist system $\mathcal{W}$ of $\mathcal{L}$-structures provides a {\df potentialist account} of a particular $\mathcal{L}$-structure $M$, if every world $W$ in $\mathcal{W}$ is a substructure of $M$ and furthermore, for every such $W$ and every individual $a\in M$, there is a world $U\in\mathcal{W}$ accessed by $W$ with $a\in U$. In other words, any world $W$ in $\mathcal{W}$ can be extended so as to accommodate any given individual of $M$. This can be seen as a weak form of directedness.\footnote{Although we have in mind the case of first-order structures and a first-order language, if one is considering structures in a higher-order logic or an infinitary logic allowing formulas with infinitely many free variables, one will want to accommodate larger collections of individuals. For example, in $L_{\omega_1,\omega_1}$ logic, one would want every world in a potentialist account to accommodate any countably many individuals in a larger world; and in a higher-order logic, one may want full directedness, so that any two worlds from $\mathcal{W}$ have a common accessible extension in $\mathcal{W}$.} We shall also say in this case that $\mathcal{W}$ {\df converges} to $M$ or that $M$ is the {\df limit} of $\mathcal{W}$. When $\mathcal{W}$ has a limit structure $M$, the worlds of $\mathcal{W}$ must all agree with $M$ on the atomic truths and therefore also with each other. Since any given world in $\mathcal{W}$ can be enlarged so as to encompass any given individual of $M$, it follows that the domain of $M$ is the union of the domains of the worlds in $\mathcal{W}$. So the limit structure $M$ is completely determined by the collection of structures $\mathcal{W}$ providing the potentialist account of it. Clearly, a given structure might have many different potentialist accounts.

Without referring directly to any limit structure, we say that a collection $\mathcal{W}$ of $\mathcal{L}$-structures is {\df coherent}, if the structures in $\mathcal{W}$ agree on atomic truth and every structure in $\mathcal{W}$ can be extended in $\mathcal{W}$ so as to accommodate any individual arising in any of the other structures. This is equivalent, of course, to saying that $\mathcal{W}$ converges to a limit structure $M$, which is simply the union of the structures of $\mathcal{W}$. For example, in the case of potential infinity, a coherent collection of $\mathcal{L}$-structures can be regarded as a collection of finite approximations of some potentially infinite limit structure. More generally, a coherent collection of structures can be regard as an approximation of some incompletable limit structure in terms of its completable parts.

For any assertion $\psi$ in the underlying language $\mathcal{L}$, consider the {\df potentialist translation} $\psi^\Diamond$, which will be the assertion in the language $\mathcal{L}^\Diamond$ arising when one replaces every instance of $\exists x$ with $\possible\exists x$ and every instance of $\forall x$ with $\necessary\forall x$. It is through this translation, in light of theorem~\ref{Theorem.Potentialist-translation}, that the potentialist can refer to truth in the limit structure, without having that structure explicitly as a part of his or her ontology~\cite{Linnebo:2013-PHS}.

\begin{theorem}\label{Theorem.Potentialist-translation}
 If $\mathcal{W}$ provides a potentialist account of a structure $M$, then truth in $M$ is equivalent to potentialist truth at the worlds of $\mathcal{W}$. Namely, for any $\mathcal{L}$-formula $\psi$ and any $a_0,\dots,a_n\in M$, we have:
  $$M\satisfies\psi(a_0,\dots,a_n)\qquad\Iff\qquad W\satisfies_{\mathcal{W}}\psi^\Diamond(a_0,\dots,a_n),$$
 for any $W\in\mathcal{W}$ in which the individuals $a_0,\dots,a_n$ exist.
\end{theorem}

\begin{proof}
This is proved by a simple induction on formulas $\psi$ in the language of $\mathcal{L}$. The claim is immediate for atomic assertions, since every $W\in\mathcal{W}$ is a submodel of $M$, and the case of Boolean connectives is also easy. For quantifiers, the point is that asserting $\exists x$ from the perspective of $M$ is exactly asserting $\possible\exists x$ from the perspective of $W$ in the potentialist semantics, because the witness $x$ might not yet exist in $W$, but if it does exist in $M$ then it will exist in some $U$ extending $W$ by the definition of what it means for $\mathcal{W}$ to provide a potentialist account of $M$. Similarly, asserting $\forall x$ in $M$ amounts to $\necessary\forall x$ in $W$ with respect to the potentialist semantics.
\end{proof}

Let us now discuss what it means for a modal assertion $\varphi(p_0,\dots,p_n)$ to be valid with respect to a potentialist system $\mathcal{W}$. In a general case, this concept applies to assertions $\varphi(p_0,\dots,p_n)$ in the language $\mathcal{L}^\Diamond(p_0,p_1,\dots)$, which expands $\mathcal{L}^\Diamond$ with propositional variables $p_0,p_1,\dots$, treated syntactically as atomic formulas. A central case for us, however, will be the validities $\varphi(p_0,\dots,p_n)$ that are expressible in the more restricted language of propositional modal logic, using only the propositional variables, logical connectives and modal operators (but no quantifiers, relations or terms). For example, we shall consider the well-known modal theories S4, S4.2, S4.3 and S5, which are expressed in propositional modal logic.

A modal assertion $\varphi(p_0,\dots,p_n)$ is {\df valid} at a world $W$ in $\mathcal{W}$ for a certain class of assertions, if all the resulting substitution instances $\varphi(\psi_0,\dots,\psi_n)$, where assertion $\psi_i$ from the allowed class is substituted for the propositional variable $p_i$, are true at $W$. For example, we shall consider cases where a modal assertion $\varphi(p_0,\dots,p_n)$ is valid for sentences in the underlying language $\mathcal{L}$, or in the potentialist language $\mathcal{L}^\Diamond$, or in these languages with parameters allowed from $W$. Let us denote by $\Val_\mathcal{W}(W,\mathcal{L})$ the set of modal propositional validities at $W$ in $\mathcal{W}$ with respect to assertions from the language $\mathcal{L}$, not allowing extra parameters from $W$; and we shall write $\Val_\mathcal{W}(W,\mathcal{L}_W)$ for the corresponding set of validities with respect to $\mathcal{L}$-assertions, where now parameters from $W$ are allowed. Clearly, it is a more severe requirement generally for a modal formula to be valid with respect to a larger collection of assertions, and so:
\begin{equation}
    \mathcal{L} \subseteq \mathcal{L}' \ \Longrightarrow \ \Val_\mathcal{W}(W, \mathcal{L}') \subseteq \Val_\mathcal{W}(W, \mathcal{L}) \notag
\end{equation}
In particular, allowing parameters results in possibly fewer validities:
    $$\Val(W,\mathcal{L}_W)\ \of\ \Val(W,\mathcal{L}).$$
We say that a propositional modal assertion $\varphi$ is {\df valid in} $\mathcal{W}$ for substitution instances from $\mathcal{L}$ if it is valid at every world of $\mathcal{W}$. Let $\Val_\mathcal{W}(\mathcal{L})$ be the resulting set of validities. It can happen that the worlds of $\mathcal{W}$ do not all exhibit the same validities, and in this case, some individual worlds will exhibit more validities than the system does as a whole. For example, in set-theoretic rank potentialism (see \S\ref{Subsection.Rank-potentialism}), the modal assertions valid at all worlds are precisely those in S4.3, but individual worlds can validate S5, which goes beyond S4.3.

In the potentialist semantics, it is natural to consider the {\df converse Barcan} formula scheme $\necessary\forall x\, p\implies\forall x\necessary p$, having substitution instances
 $$\necessary\forall x\, \psi(x)\implies\forall x\necessary \psi(x).$$
The following theorem lays out some easy potentialist validities, depending in some cases on the nature of $\mathcal{W}$.

\goodbreak
\begin{theorem}\label{Theorem.lower-bounds}
Suppose that $\mathcal{W}$ is any potentialist system of structures in a common language $\mathcal{L}$.
  \begin{enumerate}
    \item The modal theory S4 is valid at every world of $\mathcal{W}$.
    \item The converse Barcan formula is valid at every world of $\mathcal{W}$.
    \item If the accessibility relation of $\mathcal{W}$ is directed, then the modal theory S4.2 is valid at every world of $\mathcal{W}$.
    \item If the accessibility relation of $\mathcal{W}$ is linearly ordered, then the modal theory S4.3 is valid at every world of $\mathcal{W}$.
  \end{enumerate}
In each case, the validities hold for all assertions in $\mathcal{L}^\Diamond$, with parameters.
\end{theorem}

\begin{proof}
This is a standard elementary exercise in Kripke semantics, and so we leave most of the details to the reader. Statement 1 holds because the potentialist accessibility relation is reflexive and transitive; statement 2 holds because the domains of the worlds increase monotonically with the accessibility relation; statement 3 holds because the directedness of the accessibility relation leads immediately to the validity of the .2 axiom $\possible\necessary p\implies\necessary\possible p$; and statement 4 holds because linearity of the accessibility relation causes the validity of the .3 axiom $(\possible p\wedge\possible q)\implies\possible[(p\wedge\possible q)\vee(q\wedge\possible p)]$.
\end{proof}

We should like to emphasize that theorem~\ref{Theorem.lower-bounds} concerns only \emph{lower} bounds on the modal validities of a potentialist system, and it is possible in general for a potentialist system $\mathcal{W}$ to exhibit more validities than would be indicated by this theorem, as we shall discuss at length for various systems later in this article.

\bigskip\goodbreak\medskip\bigskip

\section{Some tools for studying modal validities} \label{Section.tools}

So let us now turn to the harder question of \emph{upper} bounds on the modal validities; we should like to identify useful criteria for establishing that the valid principles of potentialism for a given potentialist system $\mathcal{W}$ are contained in a certain modal theory. Using these criteria, we shall be able to determine in many cases the exact collection of modal validities for a given potentialist account. Our analysis will make use of the tools developed in~\cite{HamkinsLeibmanLoewe2015:StructuralConnectionsForcingClassAndItsModalLogic,HamkinsLoewe2008:TheModalLogicOfForcing,HamkinsLoewe2013:MovingUpAndDownInTheGenericMultiverse} for the case of the modal logic of forcing; the forcing modality was introduced in~\cite{Hamkins2003:MaximalityPrinciple}. One of the important insights of that earlier work was that the existence of certain kinds of control statements, which are easy to work with and recognize---switches, buttons, dials and ratchets---enable one to place definite upper bounds on the modal validities of a world.

To begin with an easy case, following~\cite{HamkinsLoewe2008:TheModalLogicOfForcing}, we say that an assertion $s$ is a {\df switch} in a Kripke model $\mathcal{W}$, if both $\possible s$ and $\possible\neg s$ are true at every world. A little more generally, $s$ is a switch at a particular world $W$, if $\possible s$ and $\possible\neg s$ hold in all the worlds one can reach from $W$. Thus, a switch is an assertion that can be successively turned on and off as much as one likes, by accessing new worlds in the Kripke model. A family of switches $\<s_0,s_1,\ldots>$ is {\df independent}, if one can always flip the truth values of any finitely many of the switches so as to realize any desired finite pattern of truth. If a world $W$ has independent switches, then they remain independent in any accessed world.

\begin{theorem}\label{Theorem.Switches-S5}
  If $\mathcal{W}$ is Kripke model and world $W$ admits arbitrarily large finite collections of independent switches, then the propositional modal assertions valid at $W$ are contained in the modal theory S5.
\end{theorem}

\noindent In particular, if the switches work throughout $\mathcal{W}$, as they do in many of our applications, then the validities of every world of $\mathcal{W}$ are contained within S5.

\begin{proof}
This argument follows the main idea of~\cite[theorem~10]{HamkinsLeibmanLoewe2015:StructuralConnectionsForcingClassAndItsModalLogic}, using the labeling idea of~\cite[lemma~9]{HamkinsLeibmanLoewe2015:StructuralConnectionsForcingClassAndItsModalLogic} and going back to the methods of~\cite[theorem~6]{HamkinsLoewe2008:TheModalLogicOfForcing}. Those results were stated specifically for the case of forcing and models of set theory, but the point here is that the same idea works generally in any Kripke model. So we refer the reader to that presentation, but let us also briefly explain the details. Suppose that $\mathcal{W}$ is a Kripke model which admits arbitrarily large finite collections of independent switches at world $W$. Let us consider $W$ as the initial world of $\mathcal{W}$ and assume without loss that every world of $\mathcal{W}$ can be reached from $W$ by successive applications of the accessibility relation. If $\varphi(p_0,\dots,p_n)$ is any propositional modal assertion not in S5, then it is known to fail at a world $w$ in some finite propositional Kripke model $M$ whose accessibility relation is the complete binary relation (all worlds access all worlds). By duplicating some worlds of $M$, if necessary, we may assume that the number of worlds is a power of two: $\<w_j\mid j<2^m>$. Let $\<s_k\mid k<m>$ be an independent family of $m$ switches for $\mathcal{W}$. We shall simulate a copy of $M$ inside $\mathcal{W}$ as follows. For each $j<2^m$, let $\bar s_j$ be the Boolean combination of the switches $s_k$ that accords with the same pattern as the binary bits of the number $j$. We associate each world $w_j$ in $M$ with the worlds $W$ in $\mathcal{W}$ for which $\bar s_j$ holds. We may assume that world $W$ in $\mathcal{W}$ is associated with world $w$ in $M$. Since the switch patterns $\bar s_j$ are mutually exclusive and exhaustive, this association amounts to a partition of the worlds of $\mathcal{W}$, with one piece of the partition associated with each world $w_j$ in $M$. For each propositional variable $p$ appearing in $\varphi$, let $\psi_p=\bigvee\set{\bar s_j\mid (M,w_j)\satisfies p}$. Thus, $\psi_p$ is true in a world $U\in\mathcal{W}$ just in case $p$ is true in the world $u$ of $M$ with which $U$ is associated. One may now prove by induction on formulas $\phi$ in propositional modal logic that
 $$U\satisfies_{\mathcal{W}}\phi(\psi_{p_0},\dots,\psi_{p_n})\qquad\Iff\qquad(M,u)\satisfies\phi(p_0,\dots,p_n),$$
whenever $u$ is associated with $U$, or in other words, whenever the pattern of switches $\bar s_j$ is true in $W$, where $w=w_j$. Since $\varphi$ fails at a world of $M$, we have therefore found a failing substitution instance of $\varphi$ in $\mathcal{W}$.
\end{proof}

For a related formulation, let us say that a (possibly infinite) list of statements $d_0,d_1,d_2,\dots$ is a {\df dial} in a Kripke model $\mathcal{W}$, if every world in $\mathcal{W}$ satisfies exactly one of the statements $d_i$ and furthermore, every world can access another world with any prescribed dial value. So you can adjust the dial as you like. If a world satisfies $d_i$, then we shall say that the dial value is $i$ in that world. The following argument shows how to translate between dials and independent switches. Notice that from any larger dial or from an infinite dial, we can construct smaller dials of any given size, simply by keeping any desired fewer finite number of dial statements and adding the statement that none of them holds.

\begin{theorem}\label{Theorem.Switches-iff-dials}
 A Kripke model $\mathcal{W}$ admits arbitrarily large finite families of independent switches if and only if it admits arbitrarily large finite dials.
\end{theorem}

\begin{proof}
If $\<s_i\mid i<m>$ is an independent family of switches, then for each $j<2^m$, let $d_j$ be the assertion that the switch pattern of the $s_i$ conforms exactly with the binary digits of $j$. Since every world exhibits some unique pattern for the switches, these form a mutually exclusive partition of truth; and since the switches are independent, any world can access another world realizing any given dial value $d_j$. Conversely, if $\<d_j\mid j<n>$ is a dial with $2^m\leq n$, then let $s_i$ assert that one of the $d_j$'s holds, where the $i^{\rm th}$ bit of $j$ is $1$. These are independent switches, precisely because any desired dial value is possible.
\end{proof}

Let us discuss a few more types of control statements. Following~\cite{HamkinsLoewe2008:TheModalLogicOfForcing}, a {\df button} in any Kripke model $\mathcal{W}$ is a statement $b$ such that $\possible\necessary b$ is true at every world. The button is {\df pushed} at a world if $\necessary b$ holds at that world, and otherwise unpushed. A {\df pure} button is a button $b$ for which $b\to\necessary b$ is true at every world, so that pushing the button is the same as making it become true. In S4, if $b$ is an unpushed button, then $\necessary b$ is an unpushed pure button. A family of buttons and switches is {\df independent} in a Kripke model $\mathcal{W}$, if there is a world at which the buttons are unpushed, and every world $W$ accesses a world $U$ in which any additional button may be pushed, without pushing any other as-yet unpushed buttons from the family, while also setting any finitely many of the switches so as to have any desired pattern in $U$. And similarly with dials.

\begin{theorem}\label{Theorem.Buttons+switches-S4.2}
  If $\mathcal{W}$ is a Kripke model that admits arbitrarily large finite families of independent buttons and switches, or independent buttons independent of a dial, then the propositional modal validities of $\mathcal{W}$ are contained in S4.2. The validities of any particular world in which the buttons are not yet pushed are contained in S4.2, and in any case, are contained in S5.
\end{theorem}

\goodbreak

\begin{proof}
This argument follows~\cite[theorem~13]{HamkinsLeibmanLoewe2015:StructuralConnectionsForcingClassAndItsModalLogic}, which extended the main technique of~\cite{HamkinsLoewe2008:TheModalLogicOfForcing}. That earlier result was stated only in the case of models of set theory, but the idea works generally in any Kripke model. Let us briefly describe the details. If $\varphi(p_0,\dots,p_k)$ is any propositional modal assertion not in S4.2, then by~\cite[theorem~11]{HamkinsLoewe2008:TheModalLogicOfForcing} it fails at some world in a Kripke model $M$ whose frame is a finite pre-Boolean algebra. (A pre-Boolean algebra is a partial pre-order $(B, \leq)$ whose quotient order $B / \equiv$ is a Boolean algebra, where $a\equiv b\Iff a\leq b\leq a$ is the induced equivalence relation; we shall subsequently refer to the equivalence classes as the `clusters' of the Kripke model.) We may assume without loss by duplicating worlds that the clusters of $M$ have uniform size $m$, and let $n$ be the number of atomic clusters, that is, the number of atoms in the quotient Boolean algebra. We may therefore index the worlds of $M$ as $w^a_j$, where $j<m$ and $a\of n$, where $w^a_j$ accesses $w^b_i$ just in case $a\of b$. Fix an independent family of $n$ buttons $\<b_i\mid i<n>$ and a dial $\<d_j\mid k<m>$ of length $m$ (use the switches to construct a dial if necessary). For each world $w^a_j$, let $\Phi_{w^a_j}=(\bigwedge_{i\in a}b_i)\wedge(\bigwedge_{i\notin a}\neg b_i)\wedge d_j$, which asserts that the pushed buttons are exactly those indexed in $a$ and the dial value is $j$. The independence of the buttons and the dial ensures that these assertions are mutually exclusive and a partition of truth, and furthermore that the worlds of $\mathcal{W}$ in which $\Phi_{w^a_j}$ holds can access worlds with $\Phi_{w^b_i}$ just in case $a\of b$. We associate each world $w^a_j$ of $M$ with the worlds $W\in\mathcal{W}$ in which $\Phi_{w^a_j}$ is true. For each propositional variable $p$ in $\varphi$, let $\psi_p=\bigvee\set{\Phi_{w^a_j}\mid (M,w^a_j)\satisfies p}$, which will be true in a world $W$ of $\mathcal{W}$ just in case $W$ is associated with a world $w$ in $M$ in which $p$ is true. One can now verify that
  $$W\satisfies_{\mathcal{W}}\phi(\psi_{p_0},\dots,\psi_{p_n})\qquad\Iff\qquad(M,w)\satisfies\phi(p_0,\dots,p_n),$$
whenever world $w$ in $M$ is associated with world $W$ in $\mathcal{W}$. This is because it is true for the propositional variables and is easily preserved by logical connectives, and furthermore because the accessibility relation of $M$ is exactly simulated via the association as we described. Since the original modal assertion $\varphi(p_0,\dots,p_n)$ fails at a world of $M$, we have therefore provided a substitution instance of $\varphi$ that fails at the associated worlds $W$ of $\mathcal{W}$, and so $\varphi$ is not valid in $\mathcal{W}$.

Finally, since the independent switches remain operable even after all the buttons are pushed, each individual world of $\mathcal{W}$ has its validities contained in S5 by theorem~\ref{Theorem.Switches-S5}.
\end{proof}

Following~\cite{HamkinsLeibmanLoewe2015:StructuralConnectionsForcingClassAndItsModalLogic}, a {\df ratchet} for a Kripke model $\mathcal{W}$ at a world $W$ is a sequence of assertions $r_1$, $r_2,\ldots, r_n$, such that each $r_i$ is an unpushed button in $W$, each statement $r_i$ necessarily implies all the previous statements $r_j$ for $j<i$ and over any world in which $r_i$ is not yet true, it can become true in some accessible world in which $r_{i+1}$ is not true; in this case, we say that the ratchet volume is $i$ in that accessed world. The idea of the ratchet, of course, is that it is unidirectional: the ratchet volume can only go up, never down. It is sometimes convenient to append a tautologically true statement $r_0$ at the beginning and say that the ratchet volume is $0$ in the initial world $W$. A ratchet is mutually independent with an independent family of switches, if in every world, one can achieve any finite pattern for the switches in an accessed world, without increasing the ratchet volume; and similarly with a dial.

\goodbreak

\begin{theorem}\label{Theorem.Ratchets+switches=S4.3}
  If $\mathcal{W}$ is Kripke model having arbitrarily large finite ratchets mutually independent with arbitrarily large families of independent switches (or with an arbitrarily large finite dial), then the propositional modal validities of $\mathcal{W}$ are contained in S4.3. The validities of any particular world with ratchet volume zero are contained in S4.3 and, in any case, the validities of each world are contained in S5.
\end{theorem}

\goodbreak

\begin{proof}
This argument builds on~\cite[theorems~11]{HamkinsLeibmanLoewe2015:StructuralConnectionsForcingClassAndItsModalLogic}, although again those results were stated only for models of set theory, whereas the same idea works with respect to any Kripke model. Let us briefly sketch the argument. If a modal assertion $\varphi(p_0,\dots,p_n)$ is not in S4.3, then since the finite linear pre-orders are a complete set of frames for S4.3, it follows that $\varphi$ fails at an initial world in a finite Kripke model whose frame is a finite linear pre-order (a finite chain of clusters of worlds, with each world in a cluster accessing all worlds in that cluster and any world in any higher cluster). Let $n+1$ be the number of clusters, and we may assume without loss that each cluster has uniform size $m$. So the worlds of $M$ can be described as $w^i_j$, where $i\leq n$ and $j<m$. Let $r_0,r_1,\dots,r_n$ be a ratchet in $\mathcal{W}$ of length $n$, mutually independent with a dial $\<d_j\mid j<m>$ (use switches to construct a dial if necessary). For each world $w^i_j$ in $M$, let $\Phi_{w^i_j}=r_i\wedge d_j$ be the assertion expressing that the ratchet volume is $i$ and the dial value is $j$, and we associate world $w^i_j$ with the words $W$ in $\mathcal{W}$ satisfying $\Phi_{w^i_j}$. The assumptions on the ratchet and dial ensure that a world in $\mathcal{W}$ satisfying $\Phi_{w^i_j}$ can access those worlds in $\mathcal{W}$ that satisfy some $\Phi_{w^{i'}_{j'}}$, exactly for $i\leq i'$. Thus, the $\Phi_w$ assertions partition the worlds of $\mathcal{W}$ in a way that exactly mimics the accessibility relation of $M$. For every propositional variable $p$ appearing in $\varphi$, let $\psi_p=\bigvee\set{\Phi_w\mid (M,w)\satisfies p}$, which holds exactly in the worlds of $\mathcal{W}$ that are associated with a world of $M$ in which $p$ is true. It now follows by induction on formulas that
   $$W\satisfies_{\mathcal{W}}\phi(\psi_{p_0},\dots,\psi_{p_n})\qquad\Iff\qquad(M,w)\satisfies\phi(p_0,\dots,p_n),$$
whenever $W$ is a world in $\mathcal{W}$ associated with world $w$ in $M$. Since $\varphi$ fails at an initial world of $M$, this therefore provides a substitution instance of $\varphi$ that fails at any world of $\mathcal{W}$ with zero ratchet volume. So the validities of $\mathcal{W}$ are contained amongst S4.3.

Meanwhile, because the switches remain operable even after the ratchet is used up, every individual world of $\mathcal{W}$ has its validities contained in S5 by theorem~\ref{Theorem.Switches-S5}.
\end{proof}

For the set-theoretic structures that we have in mind for our application, it often happens that one can define much longer transfinite ratchets, and these can eliminate the need for switches. So let us describe how that works.
Suppose that $\mathcal{W}$ is a potentialist system forming a coherent collection of transitive set models of some weak set theory, whose limit model is $U$. Assume furthermore that the ordinals of $U$ have limit ordinal height at least $\omega^2$ and that the models are able to carry out simple ordinal arithmetic. Generalizing the terminology of~\cite{HamkinsLeibmanLoewe2015:StructuralConnectionsForcingClassAndItsModalLogic}, let us say that $\mathcal{W}$ admits a {\df long ratchet} starting at world $W$, if there are a sequence of statements $r_\alpha$ for $\alpha\in\Ord^U$, each an unpushed button in $W$, such that each $r_\alpha$ necessarily implies all earlier $r_\beta$ for $\beta<\alpha$, no world satisfies all the $r_\alpha$ and in any world in which $r_\alpha$ is not yet true (perhaps because $\alpha$ does not yet exist in that world), then there is an accessible world in which $\alpha$ exists and $r_\alpha$ is true, but $r_{\alpha+1}$ is not true (perhaps because $\alpha+1$ does not yet exist). The ratchet is {\df uniform} if there is a formula $\eta$ such that $r_\alpha=\eta(\alpha)$, expressed with parameter $\alpha$. It is not expected necessarily that the worlds of $\mathcal{W}$ each have all the ordinals $\alpha$ of the limit model.

The main observation, originally made by George Leibman, is that a uniform long ratchet can be used as in~\cite[theorem~12]{HamkinsLeibmanLoewe2015:StructuralConnectionsForcingClassAndItsModalLogic} to provide an arbitrarily large finite ratchet mutually independent with as many finitely many independent switches as desired or with a dial. Basically, ordinals of the form $\omega\cdot k+j$ are viewed as having simulated ratchet volume $k$ and dial value $j\mod m$, or switch values given by the binary digits of $j$. Thus, one can view the assertion $r_{\omega\cdot k+j}$ as expressing simulated volume $k$ and dial value $j\mod m$, and the point is that you can change the dial value to any desired number below $m$, without increasing the simulated ratchet volume, simply by increasing $r_\alpha$ while staying in the same $\omega$-block of ordinals. For ordinals above $\omega^2$, one views them as $\lambda+j$ for some limit ordinal $\lambda$, and then takes $r_{\lambda+j}$ as asserting the top ratchet volume with dial value $j\mod m$. Thus, a uniform long ratchet provides arbitrarily long finite ratchets mutually independent with arbitrarily large dials or switches. Note that although the long ratchet $r_\alpha$ itself uses the ordinals $\alpha$ as parameters, the ratchet volume in any world is a definable ordinal (or $\Ord$ itself in that world) and we may therefore definably extract the corresponding finite ratchet and dial values without need for any parameter. We therefore may deduce:

\begin{theorem}\label{Theorem.Long-ratchet=S4.3}
  If $\mathcal{W}$ is potentialist system of transitive models of set theory admitting a uniform long ratchet, then at the worlds of $\mathcal{W}$ where the ratchet has not yet started to crank, the modal validities are contained within the modal theory S4.3, and the validities of any particular world of $\mathcal{W}$ is contained within S5.
\end{theorem}

In the previous several theorems, where we provided upper bounds on the propositional modal validities of a given Kripke model, we should like to point out that the failing substitution instances we constructed were of the form $\varphi(\psi_{p_0},\dots,\psi_{p_n})$, where the sentences $\psi_p$ were various Boolean combinations of the switch, button and ratchet assertions appearing in the hypotheses of the theorems. In particular, in the case that the Kripke model $\mathcal{W}$ is a collection of $\mathcal{L}$-structures under the potentialist semantics and the control statements are $\mathcal{L}$-sentences, then the arguments show that we achieve the upper bound modal theories S4.3, S4.2 with respect to $\mathcal{L}$-substitutions. These arguments therefore provide a slightly stronger result than achieving the upper bounds with respect to $\mathcal{L}^\Diamond$-substitutions, since otherwise the possibility might remain that there could be additional validities that work with respect to $\mathcal{L}$-substitutions, even if the stated bound is optimal for $\mathcal{L}^\Diamond$-substitutions. This observation is relevant for the applications in the next section to the case of set-theoretic potentialism, where in most cases we are able to construct the relevant control statements in the language $\mathcal{L}$ rather than only in $\mathcal{L}^\Diamond$. Nevertheless, in a few cases in \S\ref{Section.set-theoretic-potentialism} we will need the extra power of $\mathcal{L}^\Diamond$ in order to achieve the desired control statements.

\bigskip\medskip\bigskip
\section{Set-theoretic potentialism}\label{Section.set-theoretic-potentialism}

\subsection{Introduction}
Let us now turn specifically to set-theoretic potentialism, our motivating focus. We shall introduce and consider several natural kinds of set-theoretic potentialism, and for each of them, we shall use the results of the previous section in order to identify exactly what are the modal validities of that variety of potentialism, and under what circumstances a particular world might exhibit a stronger collection of validities. We are particularly interested in identifying instances of the potentialist maximality principle, which occurs when S5 is valid at a world.

It may be useful to clarify the dialectic of the article. We are analyzing various kinds of set-theoretic potentialism, but for most of the work we do so only as far as they are simulated from the perspective of \ZFC\ set theory, which we shall use as our background theory. This choice of background theory is reasonable in light of its familiarity and because it is acceptable to actualists and potentialists alike. After all, since \ZFC\ is a non-modal theory, it doesn't take a stand on whether its universe $V$ should be regarded as fully completed or as incompletable and merely potential. The former option is suggested when the language of \ZFC\ is taken at face value, while the latter is permissible when we adopt the modal translation of the language of \ZFC\ described in theorem~\ref{Theorem.Potentialist-translation}.

For each of the kinds of set-theoretic potentialism we shall consider, it turns out that the collection of modal validities can depend on a great variety of contextual choices: on the particular world that is considered, on the language that is used, on the modal vocabulary that is allowed, and on the parameters that are allowed. For example, it is possible that a world $W$ validates S5 with respect to substitution instances by sentences in the language of set theory, but only S4.3 when parameters or modal vocabulary are allowed. For this reason, it becomes somewhat fussy or technical to provide a full account of the situation; but we have nevertheless strived to do so.

The most restrictive language we consider is the language of set theory $\mathcal{L}_\in$, which beyond the logical vocabulary has only the set-membership relation $\in$; the most generous language is generally $\mathcal{L}_{\in,W}^\Diamond$, which also allows the modal operators and arbitrary parameters from the world $W$ being considered. Since as we mentioned earlier, it is a more severe requirement for a propositional modal assertion to be valid with respect to substitution instances from a more expressive language, it follows in general that we have the following inclusions for any world $W$ in any potentialist system.
$$\begin{tikzpicture}[xscale=1.5,yscale=.8]
\node at (0,0) (1) {$\Val(W,\mathcal{L}_{\in,W}^\diamond)$};
\node at (2,1) (2) {$\Val(W,\mathcal{L}_\in^\Diamond)$};
\node at (2,-1) (3) {$\Val(W,\mathcal{L}_{\in,W})$};
\node at (4,0) (4) {$\Val(W,\mathcal{L}_\in)$};
\draw (1) edge[draw=none] node [sloped, auto=false, allow upside down] {$\subseteq$} (2);
\draw (1) edge[draw=none] node [sloped, auto=false, allow upside down] {$\subseteq$} (3);
\draw (2) edge[draw=none] node [sloped, auto=false, allow upside down] {$\subseteq$} (4);
\draw (3) edge[draw=none] node [sloped, auto=false, allow upside down] {$\subseteq$} (4);
\end{tikzpicture}$$

As a preview, the following table briefly summarizes our main conclusions for the main kinds of set-theoretic potentialism that we consider. Details and additional cases can be found in the rest of the section.

\renewcommand{\descriptionlabel}[1]{\hspace\labelsep\upshape\bfseries #1}
$$\begin{minipage}{.9\textwidth}\small
\begin{description}[itemsep=4pt,style=nextline]
    \item[Rank-potentialism]
        $\mathcal{V}=\set{V_\beta\mid\beta\in\Ord}$\\
        $\possible\varphi=$ true in some larger $V_\beta$\\
        height potentialist, width actualist\\
         $S4.3\ \of\  \Val_\mathcal{V}(V_\beta,\mathcal{L}^\Diamond_{\in,V_\beta})\ \of\ \Val_\mathcal{V}(V_\beta,\mathcal{L}_\in)\ \of\  S5$\\
    \item[Grothendieck-Zermelo potentialism]
        $\mathcal{Z}=\set{V_\kappa\mid\kappa\text{ inaccessible}}$\\
        $\possible\varphi=$ true in some larger inaccessible $V_\kappa$\\
        height potentialist, width actualist\\
        $ S4.3\ \of\  \Val_\mathcal{Z}(V_\kappa,\mathcal{L}^\Diamond_{\in,V_\kappa})\ \of\    \Val_\mathcal{Z}(V_\kappa,\mathcal{L}_\in)\ \of\  S5$\\
    \item[Transitive-set potentialism]
        $\mathcal{T}=\set{W\mid W\text{ transitive}}$\\
        $\possible\varphi=$ true in some larger transitive set\\
        height potentialist, width potentialist\\
        $S4.2\ \of\ \Val_\mathcal{T}(W,\mathcal{L}_{\in,W}^\Diamond)\ \of\ \Val_\mathcal{T}(W,\mathcal{L}_\in)\ \of\  S5$\\
    \item[Forcing potentialism]
        $\mathcal{M}=\set{V[G]\mid G\text{ is }V\text{-generic}}$\\
        $\possible\varphi=$ true in some forcing extension\\
        height actualist, width potentialist\\
        $S4.2\ = \ \Val_\mathcal{M}(W,\mathcal{L}_{\in,W})\ \of \ \Val_\mathcal{M}(W,\mathcal{L}_\in)\ \of \ S5$\\
    \item[Countable-transitive model potentialism]
        $\mathcal{C}=\set{M\mid M\satisfies\ZFC,\text{ countable, transitive}}$\\
        $\possible\varphi=$ true in some larger countable transitive model of \ZFC\\
        height potentialist, width potentialist\\
        $S4.2\ =\ \Val_\mathcal{C}(W,\mathcal{L}_{\in,W})\ \of \ \Val_\mathcal{C}(W,\mathcal{L}_\in)\ \of \ S5$\\
    \item[Countable-model potentialism]
        $\mathcal{W}=\set{M\mid M\satisfies\ZFC,\text{ countable}}$\\
        $\possible\varphi=$ true in some larger countable model of \ZFC\\
        height potentialist, width potentialist\\
        $S4.3\ =\ \Val_\mathcal{W}(W,\mathcal{L}_{\in,W}^\diamond)\ = \ \Val_\mathcal{W}(W,\mathcal{L}_{\in,W})\ \of \ \Val_\mathcal{W}(W,\mathcal{L}_\in)\ \of \ S5$\\
\end{description}
\end{minipage}$$

In each case, the indicated lower and upper bounds are realized in particular worlds, usually in the strongest possible way that is consistent with the stated inclusions, although in some cases, this is proved only under additional mild technical hypotheses. Indeed, some of the potentialist accounts are only undertaken with additional set-theoretic assumptions going beyond \ZFC. For example, the Grothendieck-Zermelo account of potentialism is interesting mainly only under the assumption that there are a proper class of inaccessible cardinals, and countable-transitive-model potentialism is more robust under the assumption that every real is an element of a countable transitive model of set theory, which can be thought of as a mild large-cardinal assumption.

We shall now consider the various forms of potentialism in detail, and not only the ones mentioned above but also further variations. Since the subsections are largely independent of one another, readers may choose to focus on the forms of potentialism they happen to find most interesting.

\subsection{Set-theoretic rank potentialism}\label{Subsection.Rank-potentialism} Let us begin with a very natural case of height potentialism combined with width actualism, namely, the case of set-theoretic {\df rank potentialism}, constituted by the worlds
$$\mathcal{V}=\set{V_\beta\mid\beta\in\Ord}.$$
These worlds $V_\beta$ are the rank-initial segments of the cumulative set-theoretic hierarchy, and the potentialist modality here is ``true in some larger rank-initial segment of the universe,'' so that $\possible\varphi$ is true at some $V_\beta$, if $\varphi$ is true in some larger $V_\delta$. This system clearly provides a potentialist account of its limit, the full background set-theoretic universe $V$.

\begin{theorem}\label{Theorem.rank-potentialism-bounds}
In set-theoretic rank potentialism $\mathcal{V}$, every world $V_\beta$ obeys
\begin{equation*}
    S4.3 \ \subseteq \ \Val_\mathcal{V}(V_\beta,\mathcal{L}^\Diamond_{\in,V_\beta})  \ \of \
    \Val_\mathcal{V}(V_\beta,\mathcal{L}_\in) \ \of \ S5.
\end{equation*}
In other words, S4.3 is valid in every world for these languages and the validities of every world are contained within S5.
\end{theorem}

\begin{proof}
Since the rank-initial segments $V_\beta$ are linearly ordered, it follows by theorem~\ref{Theorem.lower-bounds} that every assertion in S4.3 is valid for rank potentialism, with respect to any language, verifying the first inclusion. The central inclusion $\Val_\mathcal{V}(V_\beta,\mathcal{L}^\Diamond_{\in,V_\beta})\of\Val_\mathcal{V}(V_\beta,\mathcal{L}_\in)$ follows from the fact that it is a less severe requirement for a propositional modal assertion to be valid with respect to less expressive language. Finally, to show the last inclusion $\Val_\mathcal{V}(V_\beta,\mathcal{L}_\in)\of S5$, it suffices by theorems~\ref{Theorem.Switches-S5} and~\ref{Theorem.Switches-iff-dials} to show that we have an infinite dial that works in all these worlds. For $j<\omega$, let $d_j$ be the assertion that the height of the ordinals is $\lambda+j$, where $\lambda$ is a limit ordinal or zero. So $d_j$ is true in $V_\beta$ if and only if $\beta=\lambda+j$ for some limit ordinal $\lambda$ or zero. These statements are expressible in the language of set theory $\mathcal{L}_\in$ without parameters or modal vocabulary. And they form a dial, because every ordinal $\beta$ is uniquely expressed as $\lambda+j$ for some limit ordinal $\lambda$ or zero and some finite $j<\omega$, and from any $V_\beta$ we can extend to a larger $V_{\lambda+j}$ so as to realize any desired $j$.
\end{proof}

\begin{theorem}\label{Theorem.Rank-potentialism-lower-bound}
The lower bound of theorem~\ref{Theorem.rank-potentialism-bounds} is sharp, for there are some worlds $V_\beta$ in the rank-potentialist system $\mathcal{V}$ with
  $$S4.3 = \Val_\mathcal{V}(V_\beta,\mathcal{L}^\Diamond_{\in,V_\beta})=
  \Val_\mathcal{V}(V_\beta,\mathcal{L}_\in).$$
\end{theorem}

\begin{proof}
Let $r_k$ be the assertion that the ordinal $\omega\cdot k$ exists. These statements are expressible in the language of set theory without parameters, since $r_k$ is equivalent in any $V_\beta$ to the assertion that there are at least $k$ distinct limit ordinals. The statements constitute a ratchet for the system of rank potentialism, since once the ordinal $\omega\cdot k$ exists in $V_\beta$, then all smaller $\omega\cdot r$ also exist; they continue to exist in all larger $V_\delta$; and if $\omega\cdot k$ does not yet exist in $V_\beta$, then we can extend $V_\beta$ to $V_{\omega\cdot k+1}$, where $\omega\cdot k$ exists but $\omega\cdot(k+1)$ does not yet exist. For any fixed $n<\omega$ and any $i<n$, let $d_i^*$ assert that the dial value $d_j$ of the previous theorem has $i=j\mod n$. That is, we use the dial values of the previous theorem modulo $n$. The point is that we can realize any dial value for $d_i^*$ without raising the ratchet volume, since we simply move to the next block of size $n$, and so we have independent ratchet and dials of any desired size, and these are expressible in the language of set theory without parameters or modal operators. It follows by theorem~\ref{Theorem.Ratchets+switches=S4.3} that the validities of $V_\beta$ for the language of set theory are precisely S4.3, for sufficiently small $\beta$, as desired.
\end{proof}

We could have alternatively proved theorem~\ref{Theorem.Rank-potentialism-lower-bound} by using a long ratchet and applying theorem~\ref{Theorem.Long-ratchet=S4.3}. Indeed, the proof we gave amounts to using the method described in theorem~\ref{Theorem.Long-ratchet=S4.3} with the ratchet $r_\alpha=$``$\alpha$ exists.''

\begin{corollary}
The modal logic of rank potentialism with respect to the language of set theory---the propositional modal assertions valid in all worlds---is exactly $\Val_\mathcal{V}(\mathcal{L}_\in) = S4.3$. Furthermore, $\Val_\mathcal{V}(\mathcal{L}^+_\in) = S4.3$ for any language $\mathcal{L}^+_\in$ extending the language of set theory $\mathcal{L}_\in$ by parameters, relations or modal vocabulary.
\end{corollary}

\begin{proof}
Theorems~\ref{Theorem.rank-potentialism-bounds} and~\ref{Theorem.Rank-potentialism-lower-bound} show that S4.3 is valid in all rank-potentialist worlds $V_\beta$, and there is some world realizing only the validities in S4.3. So it is precisely S4.3 that is valid in all worlds, and consequently $\Val_\mathcal{V}(\mathcal{L}_\in) = S4.3$. It follows by theorem~\ref{Theorem.lower-bounds} that $\Val_\mathcal{V}(\mathcal{L}_\in^+) = S4.3$ for any extension of the language, since the $V_\beta$ are linearly ordered.
\end{proof}

Let us turn now to the upper bound of theorem~\ref{Theorem.rank-potentialism-bounds}, which also is sharp. Indeed, it is sharp in a way that we find interesting and attractive, for the worlds realizing the upper bound fulfill what we call the potentialist maximality principle, an interesting set-theoretic principle of its own.

\begin{definition}\rm
Let us define that a world $W$ in a potentialist system $\mathcal{W}$ fulfills the {\df potentialist maximality principle} for a class of assertions, if the modal theory S5 is valid at $W$ for those assertions under the potentialist semantics.
\end{definition}

Since S4 is valid in any potentialist system by theorem~\ref{Theorem.lower-bounds}, the potentialist maximality principle amounts to the validity of axiom 5
 $$\possible\necessary\varphi\implies\varphi,$$
which expresses a kind of maximality principle, since it asserts that whenever assertion $\varphi$ is potentially necessarily true, then it is already actually true. Under this principle, the collection of potentially necessarily true statements that are actually true is maximized.

An ordinal $\delta$ is {\df $\Sigma_3$-correct}, if $V_\delta\elesub_{\Sigma_3} V$, meaning that $V_\delta$ and $V$ agree on the truth of $\Sigma_3$ formulas with parameters from $V_\delta$. Any such ordinal $\delta$ is a $\beth$-fixed point and a limit of such fixed points, and more. The usual proof of the \Levy-Montague reflection theorem shows that the class of all $\Sigma_3$-correct cardinals, denoted $C^{(3)}$, is closed and unbounded in the ordinals.

\goodbreak
\begin{theorem}\label{Theorem.rank-potentialism-S5-language-of-set-theory}
The following are equivalent for any ordinal $\delta$:
 \begin{enumerate}
   \item $V_\delta$ satisfies the rank-potentialism maximality principle for assertions in the language of set theory with parameters from $V_\delta$. That is,
        $$\Val_{\mathcal{V}}(V_\delta,\mathcal{L}_{\in,V_\delta})=S5.$$
   \item $\delta$ is $\Sigma_3$-correct.
 \end{enumerate}
\end{theorem}

\begin{proof}
($2\implies 1$) Assume that $\delta$ is $\Sigma_3$-correct. Since we already know that S4.3 is valid at $V_\delta$, it suffices to verify the validity of axiom 5. So suppose that $V_\delta\satisfies \possible\necessary\varphi(\vec a)$ for some assertion $\varphi$ in the language of set theory with parameters $\vec a\in V_\delta$, using the rank-potentialist semantics. This means that there is some  $\lambda\geq\delta$, such that for all $\theta\geq\lambda$ we have $V_\theta\satisfies\varphi(\vec a)$. The assertion $\exists \lambda\,\forall\theta\geq\lambda\,V_\theta\satisfies\varphi(\vec a)$ has complexity $\Sigma_3(\vec a)$ in the language of set theory\footnote{The first author has a summary blog post at~\cite{Hamkins.blog2014:Local-properties-in-set-theory}, which some readers may find helpful, concerning locally verifiable properties in set theory and the accompanying characterization of $\Sigma_2$ and $\Pi_2$ assertions.}, and so this assertion must be true in $V_\delta$, thereby verifying this instance of S5. And so the potentialist maximality principle is true at $V_\delta$ for assertions in the language of set theory with parameters from $V_\delta$.

($1\implies 2$) Conversely, assume that S5 is valid in $V_\delta$ for assertions in the language of set theory with parameters in $V_\delta$. This implies that $\delta$ is a $\beth$-fixed point, since for any ordinal $\alpha$ it is possibly necessary that $\beth_\alpha$ exists (one need only grow the universe sufficiently tall), and therefore, for any $\alpha<\delta$ we know that $\beth_\alpha$ exists already in $V_\delta$, and so $\delta=\beth_\delta$. This implies $V_\delta=H_\delta$ and so $V_\delta\elesub_{\Sigma_1} V$ by the \Levy\ reflection theorem; so $\delta$ is $\Sigma_1$-correct. A similar argument works more generally to establish $\Sigma_3$-correctness. Specifically, suppose that a $\Sigma_3$ assertion holds in $V$, using some parameters $\vec a\in V_\delta$. Every $\Sigma_3$ assertion is equivalent to an assertion of the form $\exists x\forall \beta\,V_\beta\satisfies\psi(x,\vec a)$. Thus, $V_\delta\satisfies\possible\necessary\exists x\forall\beta\,V_\beta\satisfies\psi(x,\vec a)$, since one need only grow the universe sufficiently tall until the witness $x$ is found. By S5, we may conclude that $V_\delta\satisfies\exists x\forall\beta\, V_\beta\satisfies\psi(x,\vec a)$, showing that the $\Sigma_3$ assertion holds in $V_\delta$. So $\delta$ is $\Sigma_3$-correct, as desired.
\end{proof}

\begin{corollary}
 The upper bound of theorem~\ref{Theorem.rank-potentialism-bounds} is sharp, for there are worlds $V_\delta$ with
 $$\Val_{\mathcal{V}}(V_\delta,\mathcal{L}_{\in,V_\delta})=\Val_\mathcal{V}(V_\delta,\mathcal{L}_\in)=S5.$$
\end{corollary}

\begin{proof}
  Theorem~\ref{Theorem.rank-potentialism-S5-language-of-set-theory} shows that this is true for any $\Sigma_3$-correct cardinal $\delta$, and there is a proper-class club of such cardinals.
\end{proof}

The maximality principle for the full potentialist language $\mathcal{L}_\in^\Diamond$, not just the language $\mathcal{L}_\in$ of set theory, turns out to be strictly stronger. Indeed, the following theorem shows that the full maximality principle at a world $V_\delta$ reveals that this $V_\delta$ is in a sense a miniature replica of the full ambient universe $V$ in which it sits, having all the same truths about the objects in $V_\delta$. A cardinal $\delta$ is {\df correct}, if it is $\Sigma_n$-correct for every $n$, or in other words, if it realizes the scheme $V_\delta\elesub V$. In light of Tarski's theorem on the non-definability of truth, this concept is not expressible as a single assertion in the language of set theory, although it can be expressed as a scheme of statements, the assertion that $\delta$ realizes a certain type.

\goodbreak
\begin{theorem}\label{Theorem.rank-potentialist-S5-potentialist-language}
The following schemes are equivalent for any ordinal $\delta$:
 \begin{enumerate}
   \item $V_\delta$ satisfies the rank-potentialist maximality principle for assertions in the potentialist language $\mathcal{L}_{\in,V_\delta}^\Diamond$ allowing parameters from $V_\delta$. That is,
         $$\Val_{\mathcal{V}}(V_\delta,\mathcal{L}_{\in,V_\delta}^\Diamond)=S5.$$
   \item $\delta$ is a correct cardinal.
 \end{enumerate}
\end{theorem}

\begin{proof}
($2\implies 1$) Assume that $\delta$ is correct, so that we have the scheme $V_\delta\elesub V$, and suppose that $\possible\necessary\varphi(\vec a)$ is true at $V_\delta$, where $\varphi$ is an assertion in the potentialist language $\mathcal{L}_\in^\Diamond$ and $\vec a\in V_\delta$. Thus, $\exists \lambda\forall\theta\geq\lambda\, V_\theta\satisfies\varphi(\vec a)$ is true in $V$. Since the rank-potentialist modalities are expressible in the language of set theory, it follows from the $V_\delta\elesub V$ scheme that $V_\delta\satisfies\exists\lambda\forall\theta\geq\lambda\, V_\theta\satisfies\varphi(\vec a)$. Note that since $V_\delta\elesub V$, it follows that $V_\delta$ and $V$ agree on whether some particular $V_\theta$ satisfies a given rank-potentialist assertion. Since $V_\delta\elesub V$, it follows that the witness $\lambda<\delta$ also works in $V$, and so $V_\delta\satisfies\varphi(\vec a)$, verifying this instance of S5. So $V_\delta$ satisfies the desired potentialist maximality principle.

($1\implies 2$) Conversely, suppose that S5 holds at $V_\delta$ for assertions in the potentialist language $\mathcal{L}_\in^\Diamond$ with parameters in $V_\delta$. We know from statement 1 that $\delta$ is $\Sigma_3$ correct. Suppose that $\delta$ is $\Sigma_n$-correct, and that a $\Sigma_{n+1}$ assertion is true in $V$. So $V\satisfies\exists x\,\psi(x,\vec a)$, where $\psi$ has complexity $\Pi_n$. By theorem~\ref{Theorem.Potentialist-translation}, the set-theoretic assertion $\psi(x,\vec a)$ is true in $V$ if and only if $\psi^\Diamond(x,\vec a)$ is true at any $V_\beta$ containing $x$ and $\vec a$. Thus, $V_\delta\satisfies\possible\necessary\exists x\,\psi^\Diamond(x,\vec a)$, since once one goes high enough, the witness $x$ from $V$ will exists. By S5, therefore, there is $x\in V_\delta$ with $V_\delta\satisfies\psi^\Diamond(x,\vec a)$, which means that $V\satisfies\psi(x,\vec a)$, and so we have found the desired witness $x$ inside $V_\delta$. So $\delta$ is $\Sigma_{n+1}$-correct, as desired.
\end{proof}

Although \ZFC\ proves that there are numerous $\Sigma_3$-correct cardinals $\delta$ and therefore numerous instances $V_\delta$ of the potentialist maximality principle for assertions in the language of set theory with parameters, it is meanwhile not provable in \ZFC, if consistent, that there is a cardinal $\delta$ fulfilling the scheme $V_\delta\elesub V$, and so we do not necessarily have instances of the potentialist maximality principle in the full potentialist language of set theory $\mathcal{L}_\in^\Diamond$ with parameters. Meanwhile, the $V_\delta\elesub V$ scheme is equiconsistent with and indeed conservative over \ZFC, for every model $M\satisfies\ZFC$ has an elementary extension in which this scheme is realized. To see this, consider the elementary diagram of $M$ together with the scheme asserting $V_\delta\elesub V$, in the language with $\delta$ as a new constant symbol; every finite subtheory is satisfiable by the reflection theorem applied in $M$, and so by compactness there is an elementary extension of $M$ with a correct cardinal (see also \cite[lemma~5.4]{Hamkins2003:MaximalityPrinciple}).

Note that each of the statements in theorem~\ref{Theorem.rank-potentialist-S5-potentialist-language} is a full scheme of assertions in the language of set theory. In particular, one cannot express the concept of ``$\Sigma_n$-correct for every $n$'' by a single statement in the language of set theory; instead, one makes a separate assertion for each $n$ in the metatheory. Similarly, the assertion that a modal assertion is valid for all $\mathcal{L}_\in^\Diamond$ assertions is not a single assertion in the language of set theory, but one can assert all instances of it as a scheme. In order to formalize theorem~\ref{Theorem.rank-potentialist-S5-potentialist-language} in \ZFC, therefore, we may view it as two theorem schemes, one proving every instance of statement 2, assuming the theory expressed by statement 1, and one proving every instance of statement 1, assuming the theory expressed by statement 2. This subtle formalization issue did not arise in theorem~\ref{Theorem.rank-potentialism-S5-language-of-set-theory}, because the bounded complexity of the correctness assumption there and the fact that the arbitrary set-theoretic assertions are evaluated there only in various set structures $V_\beta$ meant we could formalize it all as a single assertion in \ZFC.

We expect (and suggest as a good graduate-student project) that one can provide a somewhat tighter analysis connecting the precise degree of correctness of $\delta$ and the modal-operator complexity of the assertions allowed to be substituted in the S5 axioms. For example, if $\delta$ is $\Sigma_4$-correct, then S5 would be valid at world $V_\delta$ with respect to substitution instances in the language $\mathcal{L}_\in^\Diamond$ having at most one modal operator at the front. As the correctness of $\delta$ improves, one can accommodate more complex modal assertions in the substitution instances, and theorem~\ref{Theorem.rank-potentialist-S5-potentialist-language} is simply the amalgamation of all these level-by-level results.

\begin{corollary}
 It is relatively consistent with \ZFC\ that the upper bound of theorem~\ref{Theorem.rank-potentialism-bounds} is sharp in a stronger way, with a world $V_\delta$ having
  $$\Val_\mathcal{V}(V_\delta,\mathcal{L}_{\in,V_\delta}^\Diamond)=
  \Val_\mathcal{V}(V_\beta,\mathcal{L}_\in)=S5.$$
\end{corollary}

\begin{proof}
Theorem~\ref{Theorem.rank-potentialist-S5-potentialist-language} shows that this is true exactly at the (fully) correct cardinals $\delta$. The existence of a correct cardinal is independent of but equiconsistent with \ZFC.
\end{proof}

\subsection{Variations on set-theoretic rank potentialism}

Some height-potentialists may prefer a more restrictive class of universe fragments. For example, perhaps one wants to consider $V_\beta$ only when $\beta$ is itself a strong limit cardinal, or a $\Sigma_n$-correct cardinal, or where $V_\beta$ exhibits some other feature, such as satisfying some definable theory $T$. These situations are unified by the case where we have a proper class of ordinals $A\of\Ord$ and we consider the potentialist system
    $$\mathcal{V}_A=\set{V_\beta\mid \beta\in A}.$$

\begin{theorem}\label{Theorem.Definable-class-A-rank-potentialism}
 If $A\of\Ord$ is a definable proper class of ordinals in the set-theoretic universe $V$, and the definition of $A$ is absolute to $V_\beta$ for all $\beta\in A$, then every world $V_\beta$ in the relativized rank-potentialist system $\mathcal{V}_A$ obeys
    $$S4.3 \ \of\ \Val_\mathcal{V}(V_\beta,\mathcal{L}^\Diamond_{\in,V_\beta})  \ \of \ \Val_\mathcal{V}(V_\beta,\mathcal{L}_\in) \ \of \ S5.$$
 If $A$ is definable but not necessarily absolute, then nevertheless every world obeys at least
    $$S4.3 \ \of \Val_\mathcal{V}(V_\beta,\mathcal{L}^\Diamond_{\in,V_\beta})  \ \of \ \Val_\mathcal{V}(V_\beta,\mathcal{L}^\Diamond_\in)\ \of \ S5.$$
 Furthermore, the lower bounds in each case are realized.
\end{theorem}

\begin{proof}
Since the worlds $V_\beta$ in $\mathcal{V}_A$ are linearly ordered, it follows that every assertion of S4.3 will be valid for the potentialist semantics with respect to any of the languages. To show that some worlds validate only S4.3, it suffices by theorem~\ref{Theorem.Long-ratchet=S4.3} to find uniform long ratchets. Consider first the case where the definition of $A$ is absolute to $V_\beta$ for every $\beta\in A$. For example, this would be true if $A$ was the class of strong limit cardinals, or the class of $\Sigma_n$-correct cadinals, or if $A$ was the class of $\beta$ for which $V_\beta$ satisfied a certain specific c.e.~theory $T$. For any ordinal $\alpha$, let $r_\alpha$ be the assertion that there are at least $\alpha$ many ordinals in $A$. These statements are expressible in the language of set theory and they form a long ratchet, since once they are true in some $V_\beta$, they remain true in all taller models $V_\delta$, by the absoluteness assumption on $A$; they are all false in the smallest model $V_\beta$ where $\beta$ is the least element of $A$; they have the necessary downward implications; and for any $\alpha$, we can let $\beta$ be the $(\alpha+1)^{\rm st}$ element of $A$ and observe that $V_\beta$ thinks that there are precisely $\alpha$ many elements of $A$ (again using the absoluteness of $A$); it follows that $r_\alpha$ is true in $V_\beta$ but not $r_{\alpha+1}$. Because we have therefore produced a uniform long ratchet, expressible in the language of set theory, it follows from theorem~\ref{Theorem.Long-ratchet=S4.3} that the propositional modal validities of any sufficiently small world $V_\beta$ of $\mathcal{V}_A$, that is, where the ratchets have not yet been cranked, will be exactly
    $$S4.3=\Val_{\mathcal{V}_A}(V_\beta,\mathcal{L}_{\in,V_\beta}^\Diamond)=\Val_{\mathcal{V}_A}(V_\beta,\mathcal{L}_\in),$$
realizing the lower bound in a strong way. Furthermore, theorem~\ref{Theorem.Long-ratchet=S4.3} also shows that every world of $\mathcal{V}_A$ will have its validities contained in S5, establishing the upper bound.

Consider now the case where $A$ is definable, but not necessarily absolute to all these $V_\beta$. In this case, the statements $r_\alpha$ of the previous paragraph might not be a ratchet, since perhaps some $V_\beta$ thinks wrongly that there are $\alpha$ many elements of $A$, but a larger $V_\delta$ recognizes that many of those ordinals are not actually in $A$. Nevertheless, we can fix this problem simply by making our ratchet assertions in the potentialist modal language $\mathcal{L}_\in^\Diamond$, rather than in the language of set theory. Specifically, for our ratchet statements, we use instead the assertion $\bar r_\alpha$ that asserts that there are $\alpha$ many ordinals $\xi$ for which $(\xi\in A)^\Diamond$, using the potentialist translation of the definition of $A$ in $V$. By theorem~\ref{Theorem.Potentialist-translation}, the assertion $(\xi\in A)^\Diamond$ in some $V_\beta$ is equivalent to $\xi\in A$ in $V$, and so $\bar r_\alpha$ is true in some $V_\beta$ just in case there are at least $\alpha$ many elements of $A$ below $\beta$. (This is a sneaky trick, and the cost is that our analysis on the bounds will apply only for $\mathcal{L}^\Diamond$-substitution instances; it could be that additional validities hold when one considers only substitution instances in the language of set theory $\mathcal{L}_\in$ alone.) These revised statements in effect are able to refer via the potentialist translation to the actual class $A$, and they form a long ratchet for the same reasons that $r_\alpha$ did in the previous paragraph. Thus, in any case, the modal validities of the smallest world of $\mathcal{V}_A$ is exactly S4.3 for assertions in the potentialist language, again realizing the lower bound, and every world has its validities for that language contained in S5, establishing the upper bound.
\end{proof}

Let us turn now to the question of whether the upper bounds are realized, for which we shall undertake a similar analysis as in the case of set-theoretic rank potentialism. For any class $A$, a cardinal $\delta$ is {\df $\Sigma_3(A)$-correct}, if $\<V_\delta,\in,A\intersect V_\delta>\elesub_{\Sigma_3}\<V,\in,A>$. The proofs of theorems~\ref{Theorem.rank-potentialism-S5-language-of-set-theory} and~\ref{Theorem.rank-potentialist-S5-potentialist-language} can be adapted to establish the following.

\begin{theorem} Suppose that $A$ is a definable class of ordinals, which is absolute to $V_\delta$ for any $\delta\in A$. Then the following are equivalent for the potentialist semantics of $\mathcal{V}_A=\set{V_\beta\mid \beta\in A}$.
  \begin{enumerate}
    \item $V_\delta$ satisfies the potentialist maximality principle for assertions in the language of set theory with parameters in $V_\delta$.
    \item $\delta$ is $\Sigma_3(A)$-correct.
  \end{enumerate}
\end{theorem}

\begin{theorem}\label{Theorem.Definable-class-A-maximality-principle} Suppose that $A$ is a definable class of ordinals. Then the following are equivalent for the potentialist semantics of $\mathcal{V}_A=\set{V_\beta\mid \beta\in A}$.
  \begin{enumerate}
    \item $V_\delta$ satisfies the potentialist maximality principle for assertions in the potentialist language of set theory $\mathcal{L}_\in^\Diamond$ with parameters in $V_\delta$.
    \item $\delta$ is a correct cardinal.
  \end{enumerate}
\end{theorem}

Thus, the upper bounds of theorem~\ref{Theorem.Definable-class-A-rank-potentialism} are sharp. Notice that we do not need the predicate for $A$ in theorem~\ref{Theorem.Definable-class-A-maximality-principle}, since it is definable and we get the absoluteness of $A$ to $V_\delta$ from $\Sigma_n$-correctness, once $n$ is large enough.

Let us briefly generalize the analysis to the case of an arbtitrary class $A\of\Ord$, not necessarily definable, in \Godel-Bernays set theory \GBC. Consider the potentialist system $\mathcal{W}_A=\set{\<V_\beta,\in,A\intersect\beta>\mid\beta\in A}$ in the language of set theory with a predicate for $A$, a language we denote by $\mathcal{L}_\in(A)$.

\begin{theorem}
 Every world $W$ in $\mathcal{W}_A$ obeys
        $$S4.3\ \of\ \Val(W,\mathcal{L}^\Diamond_{\in}(A)_W)\ \of\ \Val(W,\mathcal{L}_\in(A))\ \of\  S5.$$
 The lower bound is sharp, in that there are some worlds $W$ with
        $$S4.3\ =\ \Val(W,\mathcal{L}^\Diamond_\in(A)_W)\ =\ \Val(W,\mathcal{L}_\in(A)).$$
\end{theorem}

\begin{proof}
 This is simply an adaptation of theorem~\ref{Theorem.Definable-class-A-rank-potentialism}. Since we have added $A$ explicitly to the structures, it is in effect absolute between these structures, and so we get the stronger conclusions of that theorem, but only in the language in which $A$ is a predicate.
\end{proof}

Some set-theorists may prefer to use the sets $H_\kappa$, the collection of sets of hereditary size less than $\kappa$, for a regular cardinal $\kappa$, in place of the rank-initial segments $V_\beta$, and the analysis in this case is basically the same.

\begin{theorem}
  The previous theorems also hold if one uses the sets $H_\kappa$ for regular cardinals $\kappa$, in place of the rank-initial segments $V_\beta$. In particular, the potentialist account consisting of all $H_\kappa$, or of $H_\kappa$ for regular cardinals from some definable proper class $A$, is in each case precisely S4.3.
  \begin{enumerate}
    \item S4.3 is valid for $H_\kappa$-potentialism in every world $H_\kappa$.
    \item Some worlds $H_\kappa$ validate only S4.3.
    \item Every world $H_\kappa$ has its validities contained within S5.
  \end{enumerate}
\end{theorem}

\begin{proof}
  The proof is essentially identical to the proofs of theorems~\ref{Theorem.rank-potentialism-bounds} and~\ref{Theorem.Rank-potentialism-lower-bound}. Namely, the S4.3 assertions are valid since the $H_\kappa$ are linearly ordered. Conversely, one can define ratchets for the $H_\kappa$ by letting $r_\alpha$ assert that there are $\alpha$ many regular cardinals $\kappa$ for which $H_\kappa$ is allowed as a universe fragment, using the potentialist translation as in the proof of theorem~\ref{Theorem.Definable-class-A-rank-potentialism}. So there are worlds $H_\kappa$ in which only S4.3 is valid, and meanwhile, every world has its validities within S5.
\end{proof}

Since $V_\delta=H_\delta$ for any $\beth$-fixed point $\delta$, which happens on a closed unbounded class of cardinals, the differences between $V_\delta$-potentialism and $H_\delta$-potentialism tend to evaporate once one has the idea to restrict the class of $\delta$.

\subsection{Grothendieck universe potentialism}

In current mathematical practice, the potentialist perspective is well illustrated by the category-theoretic usage of Grothendieck universes, or Grothendieck-Zermelo universes, as we shall call them. These are the rank-initial segments of the cumulative hierarchy $V_\kappa$, for an inaccessible cardinal $\kappa$. Zermelo introduced and studied these universes in 1930 (\cite{Zermelo:1930}), proving that they are exactly the models of second-order set theory $\ZFC_2$, and they have been studied by set theorists continuously since that time---they form the beginnings of the intensely studied large cardinal hierarchy. For example, the consistency strength of a single Mahlo cardinal is strictly stronger than \ZFC\ with a proper class of inaccessible cardinals. Grothendieck rediscovered these universes in the 1960s, also considering the empty universe $\emptyset$ and $V_\omega$ as instances, and used them to serve as a suitable universe concept in category theory.

In the category-theoretic practice, mathematical claims are made relative to a given universe, rather than to the entire set-theoretic universe, but one feels free at any time to move to a larger universe. This practice therefore illustrates almost perfectly the potentialist idea, and accords very well also with how Zermelo himself perceived his universes. With the second-order system $\ZFC_2$ in mind, Zermelo wrote:

\begin{quote}
%    But [the set-theoretic paradoxes] are only apparent `contradictions', and depend solely on confusing set theory itself , which is not categorically determined by its axioms, with individual models representing it.
What appears as an `ultrafinite non- or super-set' in one model is, in the succeeding model, a perfectly good, valid set with both a cardinal number and an ordinal type, and is itself a foundation stone for the construction of a new domain.~\cite{Zermelo:1930}
\end{quote}
That is, what is a proper class from the point of view of one model of $\ZFC_2$ is merely a set from the point of view of some extended model. Zermelo is thus explicitly viewing these $V_\kappa$ as set-theoretic worlds.

So let us refer to {\df Grothendieck-Zermelo (GZ) potentialism} as concerned with the potentialist system
    $$\mathcal{Z}=\set{V_\kappa\mid\kappa\text{ inaccessible}}.$$
Meanwhile, the {\df Grothendieck universe axiom} is the assertion that every set is an element of a Grothendieck-Zermelo universe, or equivalently, that the inaccessible cardinals are unbounded in the ordinals.

\begin{theorem}
 Assume the Grothendieck universe axiom holds. Then Grothendieck-Zermelo potentialism provides a potentialist account of the set-theoretic universe $V$, and every Grothendieck-Zermelo universe $W$ obeys
        $$S4.3 \ \of\ \Val_\mathcal{Z}(W,\mathcal{L}^\Diamond_{\in,W})  \ \of \ \Val_\mathcal{Z}(W,\mathcal{L}_\in) \ \of \ S5.$$
 The lower bound is sharp, in that some Grothendieck-Zermelo universes $W$ obey
        $$S4.3 \ = \Val_\mathcal{Z}(W,\mathcal{L}^\Diamond_{\in,W})  \ =\ \Val_\mathcal{Z}(W,\mathcal{L}_\in).$$
\end{theorem}

\begin{proof}
  This is an immediate consequence of theorem~\ref{Theorem.Definable-class-A-rank-potentialism}. Note that because the class of inaccessible cardinals is absolute to $V_\kappa$, we obtain the stronger conclusions of that theorem.
\end{proof}

Let us now turn to the sharpness of the upper bound, which is a statement with large cardinal strength. A cardinal $\kappa$ is {\df $\Sigma_n$-reflecting}, if it is inaccessible and $\Sigma_n$-correct.

\begin{theorem}\label{Theorem.GZ-S5-LST}
 For any inaccessible cardinal $\kappa$, the following are equivalent:
 \begin{enumerate}
   \item The Grothendieck universe axiom holds and $V_\kappa$ satisfies the GZ-potentialist maximality principle for assertions in the language of set theory with parameters from $V_\kappa$. That is, $$\Val_\mathcal{Z}(V_\kappa,\mathcal{L}_{\in,V_\kappa})=S5.$$
   \item $\kappa$ is a $\Sigma_3$-reflecting cardinal.
 \end{enumerate}
\end{theorem}

\begin{proof}
($2\to 1$) Assume that $\kappa$ is a $\Sigma_3$-reflecting cardinal. So $\kappa$ is inaccessible and $V_\kappa$ is a Grothendieck-Zermelo universe. Since $V_\kappa\elesub_{\Sigma_3}V$ and $V$ has an inaccessible cardinal, it follows that $V_\kappa$ must have an inaccessible cardinal. Similarly, for any $\beta<\kappa$ we know that $V$ has an inaccessible cardinal above $\beta$, and so $V_\kappa$ must agree. So the inaccessible cardinals are unbounded in $\kappa$. Thus, $V_\kappa$ satisfies the Grothendieck universe axiom. This axiom is itself a $\Pi_3$ assertion, and so it is true in $V$. So the Grothendieck universe axiom holds. Now suppose that $\possible\necessary\varphi(\vec a)$ holds at $V_\kappa$ in the GZ-potentialist semantics, where $\varphi$ is an assertion in the language of set theory and $\vec a\in V_\kappa$. Thus, there is some inaccessible cardinal $\lambda$ such that for all inaccessible cardinals $\theta\geq\lambda$ we have $V_\theta\satisfies\varphi$. As before, this is a $\Sigma_3$ assertion, and so it is true in $V_\kappa$. So there is a $\lambda<\kappa$, such that every inaccessible cardinal $\theta\geq\lambda$ in $V_\kappa$ has $V_\theta\satisfies\varphi(\vec a)$. The assertion that $\lambda$ has this property has complexity $\Pi_2$, and so this $\lambda$ also works in $V$. Therefore, $V_\kappa\satisfies\varphi(\vec a)$, verifying this instance of the maximality principle, as desired.

($1\to 2$) Assume that the Grothendieck universe axiom holds and that the GZ-potentialist maximality principle is true at GZ-univese $V_\kappa$. In particular, $\kappa$ is an inaccessible cardinal. So it is a $\beth$-fixed point and therefore $\Sigma_1$-correct. If a $\Sigma_2$ assertion is true in $V$, then this is witnessed in all large enough $V_\theta$, including when $\theta$ is inaccessible, and so by the maximality principle, it is witnessed in $V_\kappa$, so $\kappa$ is $\Sigma_2$-correct. Finally, suppose that some $\Sigma_3$-assertion is true. All such statements have the form $\exists x\forall y\, \psi(x,y,\vec a)$, where $\psi$ has complexity $\Sigma_1$. This implies $\possible\necessary\exists x\forall y\, \psi(x,y,\vec a)$ at $V_\kappa$, since once $\lambda$ is large enough, then $x$ will exist in $V_\theta$ for any larger $\theta$. So by the maximality principle, we conclude that $V_\kappa\satisfies\exists x\forall y\,\psi(x,y,\vec a)$, verifying this instance of $\Sigma_3$-correctness. Since $\kappa$ is inaccessible, we have proved that $\kappa$ is $\Sigma_3$-reflecting.
\end{proof}

A cardinal $\kappa$ is {\df reflecting}, if it is inaccessible and correct, so that it realizes the scheme $V_\kappa\elesub V$. In other words, it is $\Sigma_n$-reflecting for every natural number $n$. As with the correct cardinals, this is not a first-order expressible concept, although it can be expressed as a scheme.

\begin{theorem}
 For any inaccessible cardinal $\kappa$, the following are equivalent:
 \begin{enumerate}
   \item The Grothendieck universe axiom holds and $V_\kappa$ satisfies the GZ-potentialist maximality principle for assertions in the potentialist language $\mathcal{L}_\in^\Diamond$ with parameters from $V_\kappa$. That is,
       $$\Val_\mathcal{Z}(V_\kappa,\mathcal{L}_{\in,V_\kappa}^\Diamond)=S5.$$
   \item $\kappa$ is a reflecting cardinal.
 \end{enumerate}
\end{theorem}

\begin{proof}
Simply adapt the proof of theorem~\ref{Theorem.rank-potentialist-S5-potentialist-language} in the same way that theorem~\ref{Theorem.GZ-S5-LST} adapts the proof of theorem~\ref{Theorem.rank-potentialism-S5-language-of-set-theory}.
\end{proof}

\subsection{Transitive-set potentialism}

Consider next the potentialist account of $V$ arising from the class of all transitive sets
 $$\mathcal{T}=\set{W\mid W\text{ is transitive}}.$$
Thus, $\possible\psi$ is true at a transitive set $W$ if there is another transitive set $U$ with $W\of U$ and $U\satisfies_{\mathcal{T}}\psi$. Let us call this {\df transitive-set potentialism}. Since transitive sets can grow both in height and in width, this form of potentialism can be viewed as a form both of height and width potentialism. But it does not exhibit potentiality equally in height and width, since any transitive set $W$ can ultimately be completed with respect to width, by moving to the smallest $V_\beta$ containing it, and then no additional subsets of sets in $W$ will ever arise in $\mathcal{T}$. Because of this, the transitive-set potentialist account is somewhat milder with respect to  width potentialism than it is with respect to height potentialism: any set appearing in a universe fragment will eventually get all its possible subsets, after which time it will not grow in width any further, although the heights of the transitive sets never stabilize in this way, since every transitive set is included in another transitive set with additional larger ordinals. In this sense, worlds in transitive-set potentialism are width-completable in a way that they are not height-completable.

\begin{theorem}\label{Theorem.Transitive-sets-bounds}
 In transitive-set potentialism, every world $W$ obeys
    $$S4.2\ \of \ \Val_\mathcal{T}(W,\mathcal{L}_{\in,W}^\Diamond)\ \of \ \Val_\mathcal{T}(W,\mathcal{L}_\in)\ \of \ S5.$$
 The lower bound is sharp, in that some worlds $W$ have
    $$S4.2\ =\ \Val_\mathcal{T}(W,\mathcal{L}_{\in,W}^\Diamond)\ = \ \Val_\mathcal{T}(W,\mathcal{L}_\in).$$
\end{theorem}

\begin{proof}
Since the union of two transitive sets is transitive, it follows that the potentialist system of all transitive sets is directed and therefore by theorem~\ref{Theorem.lower-bounds} every assertion of S4.2 is valid. To achieve S5 as an upper bound, it suffices by theorems~\ref{Theorem.Switches-S5} and~\ref{Theorem.Switches-iff-dials} to show that this system admits an infinite dial. For this, we can use the same dial $d_j$ as in theorem~\ref{Theorem.rank-potentialism-bounds}, namely, the assertion $d_j$ that the ordinals have the form $\lambda+j$ where $\lambda$ is either a limit ordinal or zero. This is expressible by a sentence in the language of set theory, correctly interpreted inside any transitive set, and since any transitive set can be extended to any desired larger ordinal height, it forms a dial.

To show that the lower bound is sharp, it suffices by theorem~\ref{Theorem.Buttons+switches-S4.2} to find an independent family of buttons, independent of this dial. For this, let $\<B_k\mid k\in \omega>$ be an infinite list of distinct arithmetically definable infinite, co-infinite sets of natural numbers. Let $b_k$ be the assertion, ``$V_\omega$ exists and also $B_k$ exists.'' Note that this assertion is upward absolute for transitive sets, since if $V_\omega$ and $B_k$ exist in some transitive set, then this will continue to be true in any larger transitive set. So each $b_k$ is a button. Furthermore, these assertions form an independent family of buttons, since if $M$ is any transitive set, then $M\union V_\omega\union\singleton{B_k}$ is the union of transitive sets and hence also transitive, and we have added $B_k$ without adding any other $B_j$ for $j\neq k$. Since these buttons and dial values can be controlled independently of each other without interference, it follows by theorem~\ref{Theorem.Buttons+switches-S4.2} that the modal validities in any world where infinitely many of the buttons are not yet pushed will be contained within and hence equal to S4.2, as desired.
\end{proof}

The upper bound of theorem~\ref{Theorem.Transitive-sets-bounds} also is sharp, in light of the following two theorems.

\begin{theorem}\label{Theorem.transitive-set-potentialism-S5}
 The following are equivalent in transitive-set potentialism.
 \begin{enumerate}
   \item The potentialist maximality principle holds for world $M$ for assertions in the language of set theory with parameters from $M$. That is, $$\Val_\mathcal{T}(M,\mathcal{L}_{\in,M})=S5.$$
   \item $M=V_\delta$ for some $\Sigma_2$-correct cardinal $\delta$.
 \end{enumerate}
\end{theorem}

\begin{proof}
($2\implies 1$) Consider $M=V_\delta$, where $\delta$ is $\Sigma_2$-correct, and assume $\possible\necessary\varphi(\vec a)$ holds in $M$. So there is a transitive set $N$ such that in all larger transitive sets $U\fo N$, we have $U\satisfies\varphi(\vec a)$. This is a $\Sigma_2$ assertion about $\vec a$, and so it must already be true in $M$. So there is transitive set $N\in M$ such that, inside $M$, every transitive $U\fo N$ satisfies $\varphi(\vec a)$. This is a $\Pi_1$-property about $\vec a$ and this particular $N$, which therefore holds in $V$. Since $M$ itself is such a transitive set, we conclude $M\satisfies\varphi$, as desired.

($1\implies 2$) Assume that the potentialist maximality principle holds in $M$ for assertions in the language of set theory with parameters from $M$. First, we claim that $M$ is correct about power sets. If $a$ is any set in $M$, then we claim that $\possible\necessary\exists b\forall u\, u\of a\iff u\in b$. In other words, the existence of the power set of $a$ is possibly necessary. This is simply because once you move to a transitive set that has the actual power set of $a$, then it necessarily is the power set of $a$ in all larger transitive sets. So $M$ thinks $P(a)$ exists (that is, the power set relativized to $M$). If $M$ does not have the actual power set of $a$, then there is some $u\of a$ with $u\notin M$, in which case $M\union\singleton{u}$ is a transitive set containing $M$, which does not think that the power set of $a$ exists, contrary to this being necessary over $M$. So $M$ computes the power sets correctly. A similar argument shows that $M$ computes $V_\alpha$ correctly for any ordinal $\alpha\in M$, since the existence of $V_\alpha$ is possibly necessary, and $M$ cannot have a fake version of some $V_\alpha$, since in that case it would have to be wrong about some power sets. Also, for any set $a$ it is possible necessary that $a\in V_\alpha$ for some ordinal $\alpha$, and so this is already true in $M$. Thus, $M=V_\delta$ for some ordinal $\delta$. To see that $\delta$ must be $\Sigma_2$-correct, suppose that a $\Sigma_2$ assertion $\varphi(\vec a)$ is true in $V$. This is witnessed by the existence of some ordinal $\beta$ for which $V_\beta\satisfies\psi(\vec a)$ for some assertion $\psi$. So it is possibly necessary that, ``there is an ordinal $\beta$ for which $V_\beta$ exists and satisfies $\psi(\vec a)$.'' Thus, this must already be true in $M$, and so $\varphi(\vec a)$ is true already in $M$, as desired. So $\delta$ is $\Sigma_2$-correct.
\end{proof}

\begin{theorem}
 The following are equivalent in transitive-set potentialism.
 \begin{enumerate}
   \item The potentialist maximality principle holds for world $M$ for assertions in the potentialist language $\mathcal{L}_\in^\Diamond$ with parameters from $M$. That is, $$\Val_\mathcal{T}(M,\mathcal{L}_{\in,M}^\Diamond)=S5.$$
   \item $M\elesub V$. In other words, $M=V_\delta$ for a correct cardinal $\delta$, realizing the scheme $V_\delta\elesub V$.
 \end{enumerate}
\end{theorem}

\begin{proof}
 This is like the proof of theorem~\ref{Theorem.rank-potentialist-S5-potentialist-language} combined with the ideas of theorem~\ref{Theorem.transitive-set-potentialism-S5}. If the maximality principle holds for world $M$ in the potentialist language, then we know by theorem~\ref{Theorem.transitive-set-potentialism-S5} that $M=V_\delta$ for a $\Sigma_2$-correct cardinal $\delta$. But since the transitive sets form a potentialist account of $V$ in the manner of theorem~\ref{Theorem.Potentialist-translation}, we can as in theorem~\ref{Theorem.rank-potentialist-S5-potentialist-language} show that $\delta$ is $\Sigma_n$-correct for every $n$.

 Conversely, assume that $M=V_\delta\elesub V$ and $\possible\necessary\varphi(\vec a)$ holds in $M$, where $\varphi$ is a $\mathcal{L}_\in^\Diamond$ assertion. So there is a transitive set $N$, such that in all larger transitive sets $U\fo N$ we have $U\satisfies\varphi(\vec a)$. Since the existence of a such a set $N$ and the potentialist semantics are expressible in the language of set theory, it follows from $M\elesub V$ that there is such a set already in $M$, and consequently $M\satisfies\varphi(\vec a)$, as desired.
\end{proof}

Just as with rank potentialism, one may not want to allow \emph{all} transitive sets, but only some, say, those that satisfy a given c.e.~theory $T$. In this case, the collection of modal validities can depend on the theory and on the set-theoretic background.

\begin{theorem}
  Suppose that $T$ is a (sufficient) c.e.~theory and every set $x$ in $V$ is an element of a transitive model of $T$. Then the collection of all such models $$\mathcal{T}_T=\set{W\mid W\text{ is transitive and }W\satisfies T}$$ provides a potentialist account of $V$. Furthermore, every world $W$ in this system obeys
    $$S4.2\ \of \ \Val_{\mathcal{T}_T}(W,\mathcal{L}_{\in,W}^\Diamond)\ \of \
                  \Val_{\mathcal{T}_T}(W,\mathcal{L}_\in)\ \of \ S5.$$
  Reaching toward the lower bound, there is a world $W$ obeying
    $$\Val_{\mathcal{T}_T}(W,\mathcal{L}_{\in,W}^\Diamond)\ \of \
                  \Val_{\mathcal{T}_T}(W,\mathcal{L}_\in)\of S4.3.$$
  Examples show that the modal validities true at every world of $\mathcal{T}_T$ can be exactly S4.2, exactly S4.3, or exactly some other intermediate modal theory, depending on the theory $T$ and on the set-theoretic background $V$.
\end{theorem}

\begin{proof}The collection of models $\mathcal{T}_T$ provides a potentialist account of $V$, since if $W\in\mathcal{T}_T$, then $\<W,\in>$ is a substructure of $\<V,\in>$ and if $x$ is any set in $V$, then by the assumption on $T$ we have a transitive model containing $\{W,x\}$ as an element, which must therefore have $W\of U$. So the weak directedness property is fulfilled. But actually, a similar argument shows that $\mathcal{T}_T$ is fully directed: if $W_0,W_1\in\mathcal{T}_T$, then the assumption on $T$ ensures that there is $U\in\mathcal{T}_T$ with $\{W_0,W_1\}\in U$, which ensures $W_0,W_1\of U$, as desired. By theorem~\ref{Theorem.lower-bounds}, therefore, S4.2 is valid at every world for the potentialist account provided by $\mathcal{T}_T$, verifying statement 1.

Meanwhile, observe that $\mathcal{T}_T$ is well-founded with respect to $\in$, and so for any ordinal $\alpha$, we may let $r_\alpha$ be the assertion, ``the $\in$-relation on the transitive models of $T$ has rank at least $\alpha$.'' Our hypothesis in the theorem that $T$ is sufficient is meant to ensure that models of $T$ are able to refer to $T$ by means of its arithmetic definition and are also able to calculate the rank of a well-founded relation (so even some very weak set theories suffice). Since the rank can never go down as the transitive sets get larger, it follows that each $r_\alpha$ will necessarily imply all earlier $r_\beta$ for $\beta<\alpha$. And if $r_\alpha$ is not true in some transitive $W\satisfies T$, then we may simply extend $W$ to a transitive set $U\satisfies T$ of rank $\alpha+1$, so that $W\of U$ and $U\satisfies r_\alpha\wedge\neg r_{\alpha+1}$. So this is a uniform long ratchet, and therefore by theorem~\ref{Theorem.Long-ratchet=S4.3} any world in which this ratchet has not yet begun will have its validities contained within S4.3. And furthermore, the same theorem provides S5 as an upper bound for all the worlds.

Let us now give some examples showing that various modal theories can arise, depending on the theory and on the set-theoretic background. For the easiest case, if we take the empty theory, then $\mathcal{T}_T$ is the same as the class $\mathcal{T}$ of all transitive sets, whose modal validities we have already established as S4.2 in theorem~\ref{Theorem.Transitive-sets-bounds}. Next, assume $V=L$ in the background theory, and consider the theory $T$ consisting of the assertion ``V=L,'' so that the transitive models of this theory are exactly the models $L_\alpha$ appearing in the constructibility hierarchy. These are linearly ordered, and so S4.3 is valid for this potentialist system, and the validities are contained within S4.3 by the previous paragraph, yielding exactly S4.3 as the common modal theory.

Finally, assume $V=L$ in the background, and let's describe a theory $T$ for which the corresponding potentialist system $\mathcal{T}_T$ exhibits modal validities strictly in between S4.2 and S4.3. Let $\ZFC^*$ be a fixed finite fragment of $\ZFC$, which can prove the existence of the reals and which furthermore is invariant by Cohen-real forcing and by random-real forcing. Let $T$ be the theory asserting that either (i) $\ZFC^*+V=L$; or (ii) $\ZFC^*+V=L[c]$ for an $L$-generic Cohen real $c$, but there is no transitive set model of $\ZFC^*$; or (iii) $\ZFC^*+V=L[r]$ for an $L$-generic random real, but there is no transitive set model of $\ZFC^*$. For example, since $\ZFC^*$ is true in $L$, it follows by the finiteness of $\ZFC^*$ that it is true in some large $L_\gamma$, and by condensation the least $\gamma$ for which $L_\gamma\satisfies\ZFC^*$ is countable. Thus, we can easily construct $L_\gamma$-generic reals $c$ or $r$ and form the extensions $L_\gamma[c]$ and $L_\gamma[r]$. By using the least possible $\gamma$, these models will have no set models of $\ZFC^*$, and so we can see that both statements (ii) and (iii) are possible from some $L_\gamma$. Note that if one goes to a higher model, however, then (ii) and (iii) become impossible, because of the requirement that there is no transitive set model of $\ZFC^*$. In particular, both (ii) and (iii) are possible from the smallest $L_\gamma$ satisfying $\ZFC^*$, but once you have one of them, the other becomes impossible, since the forcing notions do not add generic reals of the other type. So this is exactly a violation of axiom (.3) in $L_\gamma$, and therefore the validities of this potentialist system are not S4.3. Meanwhile, we claim that the validities exceed S4.2. To see this, note that because of the homogeneity of Cohen forcing and random-real forcing, all the models of statement (ii) have the same theory, and similarly for statement (iii). Thus, every transitive model of $T$ is either a model of statement (i), and these are all linearly ordered, or else is a model of the theory of statement (ii) or of statement (iii). So it is impossible that a model of $T$ has three independent unpushed buttons. In particular, the assertion ``$p$, $q$ and $r$ are not three independent unpushed buttons,'' which is expressible in propositional modal lgoic, is valid for the potentialist system determined by the transitive models of $T$, but this assertion is not an S4.2 theorem, since some models of S4.2 do have three independent unpushed buttons. So the modal validities of $\mathcal{T}_T$ for this theory are strictly intermediate between S4.2 and S4.3, as claimed.
\end{proof}

\subsection{A remark on Solovay's set-theoretic modalities}

Before continuing, let us briefly discuss the connection of set-theoretic potentialism with some related set-theoretic modalities considered by Solovay. After proving his famous analysis of provability logic, Solovay had considered the modalities of ``true in all $V_\kappa$ for $\kappa$ inaccessible'' and also ``true in all transitive sets'' (see~\cite[chapter~13]{Boolos1993:TheLogicOfProvability} and this was extended in unpublished work of Enayat and Togha). To be precise, the intended Solovay modal semantics are that $V_\delta\satisfies\necessary\varphi$ if and only if $V_\delta$ thinks that $\varphi$ is true in all $V_\kappa$ for inaccessible $\kappa<\delta$, or $V_\delta$ thinks that $\varphi$ is true in all transitive sets in $V_\delta$, respectively.

The Solovay modalities are thus precisely the inverses of the Grothendieck-Zermelo and transitive-set potentialist modalities that we considered earlier in this article. One might more accurately describe them under the slogans ``true in all smaller $V_\kappa$ for inaccessible cardinals $\kappa$'' and ``true in all smaller transitive sets.'' The downward-oriented focus is amplified with nested modalities such as $\possible\necessary\varphi$.

This downward orientation of the Solovay modalities gives them a markedly non-potentialist character. Truly potentialist modal assertions, after all, reach outside of the current actual world to refer to objects that might potentially exist in a larger world, but which do not yet exist in the current world. Solovay's modalities are not potentialist in this way, and it is not surprising that he achieves very different modal validities than what we have found for the set-theoretic potentialist modalities in this article. The validity of the \Lob\ axiom for Solovay's modality, for example, is connected with the fact that his accessibility relation is inversely well-founded.

A separate, but related issue is that because the Kripke model for the Solovay modality has deflationary domains, one faces certain issues in the semantics of mixed quantifier/modal assertions, with the basic conundrum being the correct meaning for $\necessary\varphi(a)$ when there are accessible worlds in which $a$ does not exist (see~\cite{Garson2001:Quantification-in-modal-logic}).

\subsection{Forcing potentialism}

Let us now consider the version of set-theoretic potentialism that arises by considering the set-theoretic universe in the context of all its forcing extensions. This idea is closely connected with the modal logic of forcing, as in~\cite{HamkinsLoewe2008:TheModalLogicOfForcing, HamkinsLoewe2013:MovingUpAndDownInTheGenericMultiverse, HamkinsLeibmanLoewe2015:StructuralConnectionsForcingClassAndItsModalLogic}; the forcing modalities were introduced in~\cite{Hamkins2003:MaximalityPrinciple}. The modal logic of forcing is, at bottom, an instance of set-theoretic width potentialism and height actualism. Namely, we consider a model of set theory $M$ in the context of all its forcing extensions $M[G]$ and their further forcing extensions $M[G][H]$, interpreting $\possible\varphi$ as ``true in some forcing extension,'' and $\necessary\varphi$ as ``true in all forcing extensions.'' Since the relation of ground model to forcing extension coincides with the substructure relation amongst these models, it follows that the modal logic of forcing coincides with the potentialist semantics on this collection of models. Note that because the forcing modalities are expressible in the language of set theory, the distinction between $\mathcal{L}_\in$ and $\mathcal{L}_\in^\Diamond$ evaporates for this case. The main results of~\cite{HamkinsLoewe2008:TheModalLogicOfForcing} establish:

\begin{theorem}[{Hamkins, \Lowe}]\label{Theorem.HamkinsLowe-modal-logic-of-forcing-is-S4.2}
 In the collection of models arising as forcing extensions of a fixed countable model of \ZFC, considered under the potentialist semantics, every world $W$ obeys
    $$S4.2\ = \ \Val(W,\mathcal{L}_{\in,W})\ \of \ \Val(W,\mathcal{L}_\in)\ \of \ S5.$$
 Depending on the original model, some worlds $W$ can have
    $$S4.2\ = \ \Val(W,\mathcal{L}_{\in,W})\ = \ \Val(W,\mathcal{L}_\in).$$
 Depending on the original model, other worlds have
    $$\Val(W,\mathcal{L}_\in)=S5\quad\text{ or even }\quad\Val(W,\mathcal{L}_{\in,\R}) = S5.$$
\end{theorem}

The modal validities include S4 because inclusion is transitive and reflexive, and we get axiom .2 not because the collection of models is directed, which it isn't by the non-amalgamation results discussed in~\cite{Hamkins2016:UpwardClosureAndAmalgamationInTheGenericMultiverse, FuchsHamkinsReitz2015:Set-theoreticGeology}, but rather because any particular pair of statements $\psi$ and $\phi$ that are each separately forceably necessary, can be jointly forced by the product forcing. The main result of~\cite{HamkinsLoewe2008:TheModalLogicOfForcing} shows that in any model of $V=L$ and many others, the modal logic of forcing is exactly $\Val(W,\mathcal{L}_\in)=S4.2$, because there are independent families of buttons and switches. In the general case, the same argument shows that every world $W$ has $\Val(W,\mathcal{L}_{\in,W})=S4.2$, since there are always independent buttons and switches expressible with parameters. The same observation provides S5 generally as an upper bound, since one can express independent switches via the \GCH\ pattern without need for any parameters. The main result of~\cite{Hamkins2003:MaximalityPrinciple} shows that it is equiconsistent with \ZFC\ that S5 is valid with respect to sentences in the language of set theory (no parameters), achieving $\Val(W,\mathcal{L}_\in)=S5$; the same paper shows that $\MP(\R)$, the maximality principle allowing real parameters (that is, S5 with real parameters), has a slightly stronger consistency strength, achieving $\Val(W,\mathcal{L}_{\in,\R})=S5$.

Let us now enlarge the context by moving to the {\df generic multiverse} $\mathcal{M}$ of a model of set theory $M\satisfies\ZFC$, which is the collection of models obtained by closing under the process of forcing extensions and grounds. Because the generic multiverse of a countable model exhibits non-amalgamation (see~\cite{Hamkins2016:UpwardClosureAndAmalgamationInTheGenericMultiverse, FuchsHamkinsReitz2015:Set-theoreticGeology}), it therefore does not generally converge to a limit model. Nevertheless, we may view $\mathcal{M}$ under the potentialist semantics, and discover the modal validities of this perspective on potentialism. Indeed, the generic multiverse of a model of set theory is a natural instance of width potentialism plus height actualism, since the various models of the generic multiverse all have the same ordinals, even though sets can gain new subsets in a forcing extension. Since the modal operators are definable in the language of set theory, once again there is no distinction between $\mathcal{L}_\in$ and $\mathcal{L}_\in^\Diamond$. The following theorem was observed independently by Jakob Piribauer in his master's thesis \cite{Piribauer2017:The-modal-logic-of-generic-multiverses} at the University of Amsterdam, undertaken with supervisor Benedikt \Lowe.

\begin{theorem}\label{Theorem.Generic-multiverse-potentialism}
 In the potentialist system of the generic multiverse $\mathcal{M}$ of a fixed countable model of \ZFC\ set theory, every world $W$ obeys
    $$S4.2\ = \ \Val(W,\mathcal{L}_{\in,W})\ \of \ \Val(W,\mathcal{L}_\in)\ \of \ S5.$$
 The lower bound is sharp, for every world $M$ has a ground model $W\of M$ obeying
    $$S4.2\ = \ \Val(W,\mathcal{L}_{\in,W})\ = \ \Val(W,\mathcal{L}_\in).$$
 Depending on the original model, the upper bound also is sharp, for there can be other worlds $W$ with
    $$\Val(W,\mathcal{L}_\in)=S5\quad\text{ or even }\quad\Val(W,\mathcal{L}_{\in,\R}) = S5.$$
\end{theorem}

\begin{proof}
We've already mentioned that S4.2 is always valid for the potentialist semantics of the generic multiverse, which is the same as the modal logic of forcing for those models, and exactly S4.2 is realized for every model when parameters are allowed. Meanwhile, we also know that some models of set theory can exhibit S5 as their forcing validities, as before. In order to prove the theorem, therefore, what we shall show is that every model of set theory contains a ground whose validities are exactly S4.2, without parameters. For this, it will suffice to show that every model of set theory $M$ has a ground model $W$ by set-forcing such that $W$ has arbitrarily large finite families of independent buttons and switches, expressible without parameters. In order to prove this fact, we shall rely on the recent breakthrough result of Toshimichi Usuba~\cite{Usuba2017:The-downward-directed-grounds-hypothesis-and-very-large-cardinals}, establishing the strong downward-directed grounds hypothesis \DDG, which asserts that the ground models of the set-theoretic universe $V$ are downward set-directed: for any set-indexed family of grounds $\set{W_i\mid i\in I}$ in the uniform ground-model enumeration, there is a ground model $W$ with $W\of W_i$ for all $i\in I$. It follows that for every model of set theory $M$ and for every ordinal $\gamma$ in $M$, there is a ground model $W\of M$ such that $V_\gamma^W$ is the same as the $V_\gamma$ of the mantle of $M$, the intersection of all the grounds of $M$. In other words, $V_\gamma^W$ is least amongst the $V_\gamma$'s of the grounds of $M$. In particular, the cardinal and \GCH\ structure of $V_\gamma^W$ on the cardinals up to, say, $\aleph_{\omega+\omega}$ is absolutely definable in every model of the generic multiverse of $M$, since the mantle itself is definable in those models and $W$ agrees with the mantle in that realm.

We claim that this ground $W$ has independent buttons and switches, and therefore its valid principles of forcing are exactly S4.2. For this, we cannot seem to use the stationary-set buttons that Hamkins and \Lowe\ described in~\cite[section~4]{HamkinsLeibmanLoewe2015:StructuralConnectionsForcingClassAndItsModalLogic}, since those make reference to an absolutely definable partition using the $L$-order. Nevertheless, we may modify the buttons identified by Jakob Rittberg, or alternatively modify the buttons of Friedman, Fuchino and Sakai, but using in each case the cardinals as defined in the mantle rather than $L$. Since as we mentioned the mantle is forcing-invariant and therefore an absolutely definable class common to all the models in the generic multiverse, the referent of such expressions as $\aleph_2^{\rm Mantle}$, for example, is invariant in this generic multiverse. For this reason, the buttons and switches relativized to the mantle rather than $L$ will have the absoluteness properties that made them function as independent buttons and switches. Thus, the modal logic of forcing over $W$ is exactly $\Val(W,\mathcal{L}_\in)=S4.2$, as desired.
\end{proof}

Let us consider next an instance of set-theoretic potentialism that we find to exhibit both height potentialism and width potentialism in a natural way. Namely, consider a model of set theory $M$ in the context of its generic multiverse. It is a consequence of Usuba's theorem on the strong \DDG\ that every model $W$ in the generic multiverse of $M$ is a forcing extension of a ground of $M$, so it has the form $W=W_r[G]$, where $W_r\of M$ is the ground of $M$ indexed by $r$ in the ground-model enumeration theorem, and where $G\of\Q\in W_r$ is $W_r$-generic. Let $\mathcal{M}$ be the collection of models of the form $V_\beta^W$, where $W=W_r[G]$ is such a kind of model in the generic multiverse of $M$ and where $\beta$ is large enough so that $r,\Q,G\in V_\beta^W$. Because of non-amalgamation, $\mathcal{M}$ is not generally converging to a limit model, but we may nevertheless consider $\mathcal{M}$ under the potentialist semantics. This is height potentialism, since we may always make $\beta$ larger and thereby gain new ordinal heights; and it is width potentialism, because we can always force to add more reals or additional generic subsets to the infinite sets that we current have. What are the modal validities of this variety of set-theoretic potentialism? Let us call it the system of {\df generic-multivese rank potentialism} over $M$.

\begin{theorem}
 In the system of generic-multiverse rank potentialism over a fixed countable model of \ZFC, every world $W$ obeys
    $$S4.2\ \of \ \Val(W,\mathcal{L}_{\in,W}^\Diamond)\ \of\ \Val(W,\mathcal{L}_\in)\ \of \ S5.$$
 Depending on the original model, there can be a world $W$ with
    $$S4.2\ = \ \Val(W,\mathcal{L}_{\in,W}^\Diamond)\ =\ \Val(W,\mathcal{L}_\in).$$
\end{theorem}

\begin{proof}
Since the inclusion relation is both transitive and reflexive on $\mathcal{M}$, it follows that S4 is valid at every world, for any language interpreted in those worlds. For the validity of .2, we cannot appeal to directedness, since as we noted earlier this collection of models has non-amalgamation and therefore is not directed; nevertheless, .2 is valid as in the modal logic of forcing because of product forcing. Namely, if $\possible\necessary\varphi(\vec a)$ holds at world $V_\beta^W$, where $W=W_r[G]$ and $r,G,\Q\in V_\beta^W$, then it means there is some larger $V_\gamma^{U}$, with $U=U_s[H]$ and forcing $s,H,\P\in V_\gamma^U$, such that any still larger $V_\delta^{U'}$ satisfies $\varphi(\vec a)$. Thus, this situation is forced by some condition in $P$ over $U_s$. So for any $V_\theta^{W'}$ extending $V_\beta^W$, we can still force over $W'$ with that forcing and find a common extension $V_\lambda^{W''}$ satisfying $\varphi(\vec a)$. So S4.2 is valid.

Meanwhile, it is easy to see that the validities of any particular world are contained within S5, since one can use the \GCH\ pattern as a family of independent switches. One can always force over any given model so as to achieve any desired finite pattern for the \GCH\ at the cardinals $\aleph_n$.

Finally, consider the generic-multiverse rank-potentialist system $\mathcal{M}$ arising from a model of $\ZFC+V=L$. We may use the same independent buttons and switches of theorem~\ref{Theorem.HamkinsLowe-modal-logic-of-forcing-is-S4.2}, the point being that one doesn't need the entire universe to verify the status of the buttons and switches that were used there. So it is enough to have $V_\beta^{L[G]}$ for large enough $\beta$ in order to operate the buttons and switches.
\end{proof}

We are unsure whether every instance of generic-multiverse rank potentialism contains a world whose validities are exactly S4.2, since in our argument, we used a model of $V=L$ in order to find a model with independent buttons and switches. It would also suffice if there was a particular world $V_\beta^W$ that was tall enough and definable in all the other worlds in the system, so that we could form buttons and switches by reference to that world. In the case of generic-multiverse potentialism, we were able to use the mantle for this purpose, but in generic-multiverse rank potentialism, we do not know that the mantle will necessarily work, since perhaps some $V_\beta^W$ can be wrong about the mantle of $W$. Perhaps the mantle of $W$ is definable in $V_\beta^W$ using the language $\mathcal{L}_\in^\Diamond$? If so, then we will get that the potentialist system includes a world whose validities are S4.2 with respect to $\mathcal{L}_\in^\Diamond$-sentences.

One observation that may bear on this question is the fact that the generic-multiverse rank-potentialist system $\mathcal{M}$, consisting of the models $V_\beta^W$, where $W$ is in the generic multiverse of $M$, is identical to the forcing-extension rank-potentialist system that arises as $V_\beta^{\mathbb{M}[G]}$, where $\mathbb{M}$ is the mantle of $M$. This is true because the strong downward directedness hypothesis implies that for every ordinal $\beta$ in $M$, there is a ground $W_r\of M$ such that the $V_\beta$ of $W_r$ is the same as the $V_\beta$ of the mantle $\mathbb{M}$.

\subsection{Countable transitive model potentialism}

Let us now consider a very natural and attractive case of set-theoretic potentialism, namely, the potentialist system $\mathcal{C}$ consisting of the countable transitive models of \ZFC\ set theory.

\begin{theorem}
 In the potentialist system $\mathcal{C}$ consisting of the countable transitive models of \ZFC, every world $W$ satisfies every S4 assertion, with respect to substitution instances in any language interpreted at every world of $\mathcal{C}$; furthermore, the propositional modal validities of any particular world, with respect to substitution instances in any language extending the language of set theory, are contained in S5.
\end{theorem}

\begin{proof}
The potentialist accessibility relation is inclusion, which is transitive and reflexive. Thus, S4 is valid at every world, even with respect to expansions of the language interpreted in every world of $\mathcal{C}$. Meanwhile, if $W$ is a world in $\mathcal{C}$, then $W$ admits arbitrarily large independent families of switches, such as the assertions expressing the \GCH\ patterns on the $\aleph_n$'s. The truth of these assertions can be controlled independently by forcing over any countable transitive model of set theory. Thus, by theorem~\ref{Theorem.Switches-S5}, the modal validities of $W$ are contained within S5.
\end{proof}

The situation is much nicer and we can prove far more under a certain robustness assumption, namely, that every real is an element of a countable transitive model of \ZFC. This assumption amounts to a very weak large cardinal hypothesis, for if there is an inaccessible cardinal, or even merely a worldly cardinal, then a simple \Lowenheim-Skolem argument shows that indeed every real is an element of a countable transitive model of \ZFC.

\begin{theorem}
  Assume that every real is an element of a countable transitive model of \ZFC. Then the potentialist system $\mathcal{C}$ consisting of the countable transitive models of \ZFC\ provides a potentialist account of the collection $H_{\omega_1}$ of hereditarily countable sets. Furthermore, every world $W$ obeys
    $$S4.2\ =\ \Val_\mathcal{C}(W,\mathcal{L}_{\in,W})\ \of \ \Val_\mathcal{C}(W,\mathcal{L}_\in)\ \of \ S5.$$
  The lower bound is sharp, for some worlds $W$ have
    $$S4.2\ =\ \Val_\mathcal{C}(W,\mathcal{L}_{\in,W})\ = \ \Val_\mathcal{C}(W,\mathcal{L}_\in).$$
  The upper bound also is sharp, for some other worlds $W$ have
    $$\Val(W,\mathcal{L}^+)=S5,$$
  for any particular countable language $\mathcal{L}^+$ extending $\mathcal{L}_\in$.
\end{theorem}

\begin{proof}
Certainly every countable transitive set is a substructure of $\<H_{\omega_1},\in>$, and every element of $H_{\omega_1}$ can be coded by a real and therefore placed into a countable transitive model of \ZFC, under our assumption. It follows that $\mathcal{C}$ is directed (and even $\sigma$-directed), and so $\mathcal{C}$ provides a potentialist account of $H_{\omega_1}$. Since this system is directed, it follows that S4.2 is valid at every world, regardless of what other structure is imposed or which language or parameters are used.

Let $L_\alpha$ be the smallest transitive model of \ZFC. This world has independent families of buttons and switches, using a modified version of those in the modal logic of forcing. Namely, the button $b_n$ asserts that $n^{\rm th}$ piece of the $L$-least partition of $\omega_1^{L_\alpha}$ is no longer stationary (note that $L_\alpha$ is definable as a class in any transitive model of \ZFC), and the switch $s_m$ asserts that the \GCH\ holds at (the current) $\aleph_m$, for $m\geq 1$. These buttons and switches can be controlled independently by forcing, although the buttons are indeed buttons with respect to the potentialist semantics. In particular, moving to a taller model will push all the buttons, since $L_\alpha$ will be realized as countable there. Since we have sufficient independent buttons and switches, it follows by theorem~\ref{Theorem.Buttons+switches-S4.2} that the validities of this world, with respect to substitution instances in the language of set theory, are exactly $\Val(L_\alpha,\mathcal{L}_\in)=S4.2$. The same idea works in any model, if parameters are allowed, showing $S4.2=\Val(W,\mathcal{L}_{\in,W})$ for every world. Meanwhile, the \GCH-pattern switches are operable over any model of set theory, and so the validities of any particular world are contained within S5.

The upper bound is provided by theorem~\ref{Theorem.CTM-maximality-principle}.
\end{proof}

Since $\mathcal{C}$ provides a potentialist account of $H_{\omega_1}$, it follows from theorem~\ref{Theorem.Potentialist-translation} that in the potentialist language $\mathcal{L}^\Diamond$ we may at any world $W\in\mathcal{C}$ to refer to truth in $H_{\omega_1}$, including truth about what other kinds of worlds in $\mathcal{C}$ there may be and how they are related.

Note also that many other worlds of $\mathcal{C}$ than the Shepherdson-Cohen model $L_\alpha$ will also have validities limited to S4.2. For example, any world $W$ that is absolutely definable in all larger worlds will also have independent buttons and switches, for essentially the same reason. This includes many models with various large cardinals, if any exist, such as the smallest transitive model $L_\gamma[\mu]$ realizing the smallest possible ordinal as a measurable cardinal, and similarly with much stronger large cardinal notions.

Let us now turn to the upper bound, by investigating the nature of the worlds satisfying the potentialist maximality principle S5 for this version of potentialism. The maximality principle occurs when S5 is valid at a world. We find these worlds quite attractively to instantiate the \emph{maximize} maxim, with respect both to height and width potentialism. The idea is that if $W$ satisfies the potentialist maximality principle, then any statement $\sigma$ that could become necessarily true, either by moving to a wider or to a taller universe, is already true.

\begin{theorem}\label{Theorem.CTM-maximality-principle}
  If every real is an element of a countable transitive model of \ZFC, then every world $U\in\mathcal{C}$ can be extended to a world $W\in\mathcal{C}$ satisfying the potentialist maximality principle for assertions in any fixed countable language $\mathcal{L}^+$ extending $\mathcal{L}_\in^\Diamond$ (interpreted in every model of $\mathcal{C}$). That is,
    $$\Val_\mathcal{C}(W,\mathcal{L}^+)=S5.$$
\end{theorem}

\begin{proof}
Assume that every real is an element of a countable transitive model of \ZFC. It follows as we noted earlier that any countably many worlds of  $\mathcal{C}$ have a common upper bound. Consider any particular world $U\in\mathcal{C}$ and any countable language $\mathcal{L}^+$ extending the language of set theory. In particular, the new language might have constants for every element of $U$. Enumerate the assertions of $\mathcal{L}^+$ as $\sigma_0$, $\sigma_1$, and so on. Let $W_0=U$ be the initial world. Given $W_n$, ask whether $W_n\satisfies\possible\necessary\sigma_n$ in this potentialist system. If so, then there is some $W_{n+1}\satisfies\necessary\sigma_n$; otherwise, choose $W_{n+1}$ extending $W_n$ as you like. Having selected $W_n$ for all $n<\omega$, let $W$ be any world with $W_n\of W$ for all $n$. If $W\satisfies\possible\necessary\sigma_n$, then since $W_n$ can access $W$, it follows that $W_n\satisfies\possible\necessary\sigma_n$, in which case $W_{n+1}\satisfies\necessary\sigma_n$, and so $W\satisfies\sigma_n$. Thus, $W$ satisfies every instance of axiom 5 in this language, and since it also satisfies S4.2, it follows that S5 is valid in $W$. And the same conclusion holds with respect to any larger world than $W$, since S5 is persistent.
\end{proof}

We find it interesting to consider the previous theorem in the case that $V=L$. In this situation $L$ can have countable transitive models of some very rich theories, such as a proper class of Woodin cardinals or supercompact cardinals or what have you. So not every $W\in\mathcal{C}$ will satisfy $V=L$, and some will believe themselves to be very far from $L$. Nevertheless, because every such $W$ is countable in $L$, we will be able to extend $W$ to some $L_\alpha$, which of course satisfies $V=L$. In particular, this shows that we can achieve the CTM-potentialist maximality principle in a world satisfying $V=L$. This situation illustrates one of the central philosophical points made by the first author in~\cite{Hamkins2014:MultiverseOnVeqL}, namely, that the capacity to extend the universe upwards tends to undermine the view of $L$ as contradicting the `maximize' maxim of Maddy.

In particular, if $V=L$ and every real is an element of a countable transitive model of \ZFC, then there are such worlds $W$ satisfying the potentialist maximality principle for sentences in any countable language extending the language of set theory (and including $\mathcal{L}_\in^\Diamond$ if desired), such that $W\satisfies V=L$; and there are other such worlds $W$ satisfying the potentialist maximality principle for which $W\satisfies V\neq L$.

\subsection{Countable model potentialism}

Let us now enlarge the potentialist system from merely the transitive models to the collection consisting of all countable models of \ZFC\ set theory.

At first, let's consider this as a potentialist-like system under the substructure relation, so that one world $\<W,\in^W>$ accesses another $\<U,\in^U>$ just in case the first is a substructure of the second, which means $W\of U$ and $\in^W=\in^U\restrict W$. The potentialist idea here is that one can move to a larger universe of set theory by adding more objects, including new elements of old sets, as long as the new $\in$-relation agrees with the old one on the previously existing sets.

\begin{theorem}
  Assume \ZFC\ is consistent and consider the potentialist system consisting of all countable models of \ZFC, under the substructure relation. Every world $W$ in this potentialist system obeys
    $$S4.3\ =\ \Val(W,\mathcal{L}_{\in,W}^\diamond)\ = \ \Val(W,\mathcal{L}_{\in,W})\ \of \ \Val(W,\mathcal{L}_\in)\ \of \ S5.$$
  If $W$ is any countable nonstandard model of \ZFC, then
    $$\Val(W,\mathcal{L}_\in^\Diamond)\ =\ \Val(W,\mathcal{L}_\in)=S5.$$
\end{theorem}

\begin{proof}
The main theorem of~\cite{Hamkins2013:EveryCountableModelOfSetTheoryEmbedsIntoItsOwnL} shows that the countable models of \ZFC\ are linearly pre-ordered by embeddability: for any two models, one of them is isomorphic to a substructure of the other, and furthermore, one model embeds into another just in case the ordinals of the first order-embed into the ordinals of the second. It follows that the nonstandard models of \ZFC\ are universal under embeddability.

Because the models are linearly pre-ordered by embeddability, it follows that S4.3 is valid in every world. Because these are \ZFC\ models, the \GCH-pattern statements are independent switches, and so the validities of any world are contained in S5. Meanwhile, for any world $W$, let $a=\emptyset^W$, and let $r_n$ be the statement ``$a$ has at least $n$ distinct elements.'' These statements are all false in $W$, of course, since $a$ has no elements in $W$. But the particular embeddings constructed in~\cite{Hamkins2013:EveryCountableModelOfSetTheoryEmbedsIntoItsOwnL} and variations of them show that there are larger models $U$ into which $W$ embeds, in which as many new elements of $a$ as desired are introduced. Thus, these statements form a ratchet, and the \GCH-pattern statements are a mutually independent family of switches. So the modal validites of $W$ for such statements are contained within S4.3, and so the first chain of equations and inclusions is proved.

Finally, for the final statement of the theorem, it is a consequence of the fact that the nonstandard models all access isomorphic copies of each other---the main theorem of~\cite{Hamkins2013:EveryCountableModelOfSetTheoryEmbedsIntoItsOwnL} shows that they are all bi-embeddable---that S5 is valid in each of them for sentences of $\mathcal{L}_\in^\Diamond$.
\end{proof}

Note that a simple compactness argument shows that this potentialist system has amalgamation and even $\sigma$-amalgamation, and furthermore has the joint embedding property, and so actually there is a \Fraisse\ limit here, which will be a copy of the countable random $\Q$-graded digraph, as described in~\cite{Hamkins2013:EveryCountableModelOfSetTheoryEmbedsIntoItsOwnL}. Even though the collection of countable \ZFC\ models does not have a limit model in the sense of our original definition, if we used this embedding-based approach to defining the limit structure, we would be led to find a limit model in this countable random $\Q$-graded digraph. The substructures of this graph that are \ZFC\ models form a potentialist system containing isomorphic copies of any given countable model of \ZFC, which furthermore provide a potentialist account of this graph as the limit structure.

Some set theorists may object that it is more natural to use a slightly more restricted accessiblity relation, because the substructure relation allows for some rather strange things. For example, with the substructure accessibility relation, every set is potentially finite, because results in~\cite{Hamkins2013:EveryCountableModelOfSetTheoryEmbedsIntoItsOwnL} show that every model is isomorphic to a substructure of the hereditarily finite sets of any $\omega$-nonstandard model. Further, no sets are necessarily disjoint, since for any two sets $a,b$, we can access a larger world that has an object that is an element of both. Similarly, the empty set of one world can have members or even become infinite in another world. All these things seem a little strange.

So it may be a little more orderly or stable to consider the collection of countable models of \ZFC\ under the \emph{transitive-substructure} relation, rather than the full substructure relation. That is, in this modified version, we say that one world $\<W,\in^W>$ accesses another world $\<U,\in^U>$, just in case the first is a transitive substructure of the second, meaning that $W\of U$ and $\in^W=\in^U\restrict W$ and furthermore, $x\in^U y\in W$ implies $x\in W$. So the larger world agrees with the smaller world on the membership relation of objects in the smaller world and furthermore, no new elements of old sets are added.

Since this modified accessibility relation is reflexive and transitive, and refines the substructure relation, this is indeed a potentialist system, which therefore validates S4 at every world. And since these are models of set theory, the \GCH-pattern switches show that the validities of any particular world are contained within S5, even when restricted to sentences in the language of set theory.

This modified system is not directed, nor even weakly directed, because if $\<W,\in^W>$ is a countable ill-founded model of \ZFC, then there is a countable $\in^W$-descending sequence $s$ and there can be no extension $\<U,\in^U>$ containing $s$ in which $\<W,\in^W>$ is transitive, even if $W$ is an $\omega$-model, since $s$ reveals the ill-foundedness of $W$. Thus, this system does not provide a potentialist account of any limit model.

We believe that under some reasonable assumption, we can get that this potentialist system has upper bounds for countable increasing chains in the transitive substructure relation, as in the result of~\cite{Hamkins2016:UpwardClosureAndAmalgamationInTheGenericMultiverse}. That is, if you can move successively to larger and larger countable models of ZFC, with each model transitive in the next one, then there is a model at the top inside of which they are all transitive submodels. In this case, we can get models of S5 in this system, for sentences in any countable language, by the same argument as in theorem~\ref{Theorem.CTM-maximality-principle}. We shall leave further investigation of this modified potentialist system for another project.

We suggest that it may be fruitful to consider the multiverse model of Gitman and Hamkins~\cite{GitmanHamkins2010:NaturalModelOfMultiverseAxioms} as a potentialist system under either of the accessibility relations mentioned above. The collection of worlds are the countable computably-saturated models of \ZFC, and we may consider them either under the substructure accessibility relation or the transitive-substructure accessibility relation.

Lastly, we should like to mention that in current work, the first author  \cite{Hamkins:The-modal-logic-of-arithmetic-potentialism} has established that the modal logic of end-extensional potentialism for the models of \PA\ is precisely S4. Adapting that work to the case of set theory, he and W. Hugh Woodin \cite{HamkinsWoodin:The-universal-finite-set} similarly established that the modal logic of top-extensional potentialism for the countable models of \ZFC\ is also precisely S4.

\printbibliography

\end{document}

@BOOK{Hellman:1989,
  AUTHOR =       {Geoffrey Hellman},
  TITLE =        {Mathematics without Numbers},
  PUBLISHER =    {Clarendon},
  YEAR =         {1989},
  address =      {Oxford},
}

@ARTICLE{Linnebo:2013-PHS,
  AUTHOR =       {{\O}ystein Linnebo},
  TITLE =        {The Potential Hierarchy of Sets},
  YEAR =         {2013},
  JOURNAL =      {Review of Symbolic Logic},
  volume =       {6},
  number =       {2},
  pages =        {205-228},
}
@INCOLLECTION{Parsons:1983b,
  AUTHOR =       {Charles Parsons},
  TITLE =        {{Sets and Modality}},
  BOOKTITLE =    {Mathematics in Philosophy},
  PUBLISHER =    {Cornell University Press},
  YEAR =         {1983},
  pages =        {298-341},
  address =      {Cornell, NY},
}
@article{Studd_ItConceptBiModal,
	title = {The Iterative Conception of Set: A (Bi-)Modal Axiomatisation},
	journal = {Journal of Philosophical Logic},
	year = {2013},
	pages = {697--725},
	volume = {42},
	number = {5},
	author = {James Studd}
}

@ARTICLE{Zermelo:1930,
  AUTHOR =       {Ernst Zermelo},
  TITLE =        {{\"{U}ber Grenzzahlen und Mengenbereiche}},
  JOURNAL =      {Fundamenta Mathematicae},
  YEAR =         {1930},
  volume =       {16},
  pages =        {29-47},
  note =         {Translated in~\cite{Ewald:1996}},
}
@BOOK{Ewald:1996,
  AUTHOR =       {William Ewald},
  TITLE =        {From Kant to Hilbert: A Source Book in the Foundations of Mathematics},
  PUBLISHER =    {Oxford University Press},
  YEAR =         {1996},
  volume =       {2},
  address =      {Oxford},
}

@ARTICLE{Putnam:1967b,
  AUTHOR =       {Hilary Putnam},
  TITLE =        {{Mathematics without Foundations}},
  JOURNAL =      {Journal of Philosophy},
  YEAR =         {1967},
  volume =       {LXIV},
  number =       {1},
  pages =        {5-22},
}

@article{Lear1977-setssem,
	year = {1977},
	author = {Jonathan Lear},
	pages = {86--102},
	number = {2},
	title = {Sets and Semantics},
	journal = {Journal of Philosophy},
	volume = {74},
}
@MISC{Linnebo&Shapiro:ActPotInf,
  AUTHOR =       {{\O}ystein Linnebo and Stewart Shapiro},
  TITLE =        {Actual and potential infinity},
  YEAR =         {2016},
  howpublished = {Unpublished manuscript},
}
@incollection{Tait1998-ZermeloConceptionSets,
	year = {1998},
	author = {W. W. Tait},
	publisher = {Clarendon Press},
	booktitle = {The Philosophy of Mathematics Today},
	editor = {Matthias Schirn},
	title = {Zermelo's Conception of Set Theory and Reflection Principles}
}

%% file: MathMacrosJDH.tex
\newtheorem{theorem}{Theorem}

\newtheorem*{maintheorem*}{Main Theorem}
\newtheorem*{maintheorems*}{Main Theorems}
\newtheorem{corollary}[theorem]{Corollary}
\newtheorem*{corollary*}{Corollary}
\newtheorem*{corollaries*}{Corollaries}

\newtheorem*{question*}{Question}

\newtheorem*{questions*}{Questions}
\newtheorem*{mainquestion*}{Main Question} % without numbering
\newtheorem*{openquestion*}{Open Question} % without numbering

\newtheorem{definition}[theorem]{Definition}

\newcommand{\QED}{\end{proof}}

\def\proclaim[#1]{{\bf #1}}
\def\BF#1.{{\bf #1.}}

\def\says#1:#2\par{\item[#1] #2\par}
\newcommand{\Lob}{L\"ob}

\newcommand{\Fraisse}{Fra\"\i ss\'e}

\newcommand{\Godel}{G\"odel}

\newcommand{\Lowe}{L\"owe}

\newcommand{\Levy}{L\'{e}vy}
\newcommand{\Lowenheim}{L\"owenheim}
\newcommand{\Oystein}{{\O}ystein}

\renewcommand{\P}{{\mathbb P}}
\newcommand{\Q}{{\mathbb Q}}

\newcommand{\R}{{\mathbb R}}

\makeatletter
\newcommand{\dotminus}{\mathbin{\text{\@dotminus}}}
\newcommand{\@dotminus}{%
  \ooalign{\hidewidth\raise1ex\hbox{.}\hidewidth\cr$\m@th-$\cr}%
}
\makeatother

\newcommand{\of}{\subseteq}

\newcommand{\fo}{\supseteq}

\newcommand{\set}[1]{\{\,{#1}\,\}}
\newcommand{\singleton}[1]{\left\{{#1}\right\}}

\newcommand{\elesub}{\prec}

\newcommand{\restrict}{\upharpoonright} % uses amssymb
\newcommand{\satisfies}{\models}

\DeclareMathOperator{\possible}{\text{\tikz[scale=.6ex/1cm,baseline=-.6ex,rotate=45,line width=.1ex]{\draw (-1,-1) rectangle (1,1);}}}
\DeclareMathOperator{\necessary}{\text{\tikz[scale=.6ex/1cm,baseline=-.6ex,line width=.1ex]{\draw (-1,-1) rectangle (1,1);}}}

\newcommand{\union}{\cup}

\newcommand{\intersect}{\cap}

\newcommand{\smalllt}{\mathrel{\mathchoice{\raise2pt\hbox{$\scriptstyle<$}}{\raise1pt\hbox{$\scriptstyle<$}}{\raise0pt\hbox{$\scriptscriptstyle<$}}{\scriptscriptstyle<}}}
\newcommand{\smallleq}{\mathrel{\mathchoice{\raise2pt\hbox{$\scriptstyle\leq$}}{\raise1pt\hbox{$\scriptstyle\leq$}}{\raise1pt\hbox{$\scriptscriptstyle\leq$}}{\scriptscriptstyle\leq}}}

\newcommand{\boolval}[1]{\mathopen{\lbrack\!\lbrack}\,#1\,\mathclose{\rbrack\!\rbrack}}
\def\[#1]{\boolval{#1}}
\newbox\gnBoxA
\newdimen\gnCornerHgt
\setbox\gnBoxA=\hbox{\tiny$\ulcorner$}
\global\gnCornerHgt=\ht\gnBoxA
\newdimen\gnArgHgt
\def\gcode #1{%
\setbox\gnBoxA=\hbox{$#1$}%
\gnArgHgt=\ht\gnBoxA%
\ifnum     \gnArgHgt<\gnCornerHgt \gnArgHgt=0pt%
\else \advance \gnArgHgt by -\gnCornerHgt%
\fi \raise\gnArgHgt\hbox{\tiny$\ulcorner$} \box\gnBoxA %
\raise\gnArgHgt\hbox{\tiny$\urcorner$}}
\newcommand{\UnderTilde}[1]{{\setbox1=\hbox{$#1$}\baselineskip=0pt\vtop{\hbox{$#1$}\hbox to\wd1{\hfil$\sim$\hfil}}}{}}
\newcommand{\Undertilde}[1]{{\setbox1=\hbox{$#1$}\baselineskip=0pt\vtop{\hbox{$#1$}\hbox to\wd1{\hfil$\scriptstyle\sim$\hfil}}}{}}
\newcommand{\undertilde}[1]{{\setbox1=\hbox{$#1$}\baselineskip=0pt\vtop{\hbox{$#1$}\hbox to\wd1{\hfil$\scriptscriptstyle\sim$\hfil}}}{}}
\newcommand{\UnderdTilde}[1]{{\setbox1=\hbox{$#1$}\baselineskip=0pt\vtop{\hbox{$#1$}\hbox to\wd1{\hfil$\approx$\hfil}}}{}}
\newcommand{\Underdtilde}[1]{{\setbox1=\hbox{$#1$}\baselineskip=0pt\vtop{\hbox{$#1$}\hbox to\wd1{\hfil\scriptsize$\approx$\hfil}}}{}}

\renewcommand{\implies}{\mathrel{\rightarrow}}

\renewcommand{\iff}{\mathrel{\leftrightarrow}}
\newcommand{\Iff}{\mathrel{\Longleftrightarrow}}

\def\<#1>{\left\langle#1\right\rangle}

\newcommand{\Ord}{\mathord{{\rm Ord}}}
\newcommand{\ZFC}{{\rm ZFC}}

\newcommand{\GBC}{{\rm GBC}}

\newcommand{\GCH}{{\rm GCH}}

\newcommand{\DDG}{{\rm DDG}}

\newcommand{\MP}{{\rm MP}}

\newcommand{\PA}{{\rm PA}}

\newcommand{\cell}[1]{\boxit{\hbox to 17pt{\strut\hfil$#1$\hfil}}}
\newcommand{\head}[2]{\lower2pt\vbox{\hbox{\strut\footnotesize\it\hskip3pt#2}\boxit{\cell#1}}}
\newcommand{\boxit}[1]{\setbox4=\hbox{\kern2pt#1\kern2pt}\hbox{\vrule\vbox{\hrule\kern2pt\box4\kern2pt\hrule}\vrule}}
\newcommand{\Col}[3]{\hbox{\vbox{\baselineskip=0pt\parskip=0pt\cell#1\cell#2\cell#3}}}
\newcommand{\tapenames}{\raise 5pt\vbox to .7in{\hbox to .8in{\it\hfill input: \strut}\vfill\hbox to
.8in{\it\hfill scratch: \strut}\vfill\hbox to .8in{\it\hfill output: \strut}}}
\newcommand{\Head}[4]{\lower2pt\vbox{\hbox to25pt{\strut\footnotesize\it\hfill#4\hfill}\boxit{\Col#1#2#3}}}
\newcommand{\Dots}{\raise 5pt\vbox to .7in{\hbox{\ $\cdots$\strut}\vfill\hbox{\ $\cdots$\strut}\vfill\hbox{\
$\cdots$\strut}}}
\newcommand{\df}{\it} % use italic for definition terms. Idea: also use this to create an index of definitions, if MakeIndex is true.